\documentclass{amsart}

\usepackage[utf8]{inputenc}
\usepackage{amsmath,amssymb,amsthm,bbm,centernot,xcolor,tikz,upgreek,esint,float,comment,mathrsfs,setspace}
\usepackage{hyperref}
\usepackage[shortlabels]{enumitem}

\newcommand{\e}{\epsilon}
\renewcommand{\gamma}{\upgamma}
\newcommand{\N}{\mathbb{N}}
\newcommand{\Z}{\mathbb{Z}}

\newcommand{\cQ}{\mathcal{Q}}
\newcommand{\R}{\mathbb{R}}

\renewcommand{\d}{\delta}

\newcommand{\cH}{\mathcal{H}}

\renewcommand{\R}{\mathbb{R}}
\newcommand{\cR}{\mathcal{R}}
\newcommand{\cL}{\mathcal{L}}

\renewcommand{\S}{\mathbb{S}}

\renewcommand{\d}{\mathrm{d}}

\DeclareMathOperator\diam{diam}
\DeclareMathOperator\dist{dist}
\DeclareMathOperator\card{card}

\DeclareMathOperator\Mod{Mod}
\DeclareMathOperator\Capa{Cap}
\DeclareMathOperator\len{length}

\newcommand{\gsim}{\gtrsim}

\newcommand{\comm}[1]{\ignorespaces}

\newtheorem{Thm}{Theorem}[section]
\newtheorem*{Thm*}{Theorem}
\newtheorem{Prop}[Thm]{Proposition}
\newtheorem*{Prop*}{Proposition}

\newtheorem{Lem}[Thm]{Lemma}
\newtheorem*{Lem*}{Lemma}
\newtheorem{thm}{Theorem}[section]
\newtheorem*{thm*}{Theorem}
\newtheorem{prop}[thm]{Proposition}
\newtheorem*{prop*}{Proposition}
\newtheorem{cor}[thm]{Corollary}
\newtheorem{lem}[thm]{Lemma}
\newtheorem*{lem*}{Lemma}
\newtheorem*{Cla*}{Claim}
\theoremstyle{remark}

\newtheorem*{rem*}{Remark}

\theoremstyle{definition}

\newtheorem*{Def*}{Definition}

\numberwithin{equation}{section}

\title{Bi-Lipschitz arcs in metric spaces with controlled geometry}

\author{Jacob Honeycutt}

\author{Vyron Vellis}

\author{Scott Zimmerman}

\thanks{V.V. was partially supported by NSF DMS grants 1952510 and 2154918.}
\date{\today}

\address{Department of Mathematics\\ The University of Tennessee\\ Knoxville, TN 37966}
\email{jhoney12@vols.utk.edu}
\address{Department of Mathematics\\ The University of Tennessee\\ Knoxville, TN 37966}
\email{vvellis@utk.edu}
\address{Department of Mathematics, The Ohio State University at Marion, 1465 Mt Vernon Ave, Marion, OH 43302}
\email{zimmerman.416@osu.edu}

\begin{document}
\maketitle

\begin{abstract}
We generalize a bi-Lipschitz extension result of David and Semmes from Euclidean spaces to complete metric measure spaces with controlled geometry (Ahlfors regularity and supporting a Poincar\'e inequality). In particular, we find sharp conditions on metric measure spaces $X$ so that any bi-Lipschitz embedding of a subset of the real line into $X$ extends to a bi-Lipschitz embedding of the whole line. Along the way, we prove that if the complement of an open subset $Y$ of $X$ has small Assouad dimension, then it is a uniform domain. Finally, we prove a quantitative approximation of continua in $X$ by bi-Lipschitz curves.
%
\end{abstract}


\section{Introduction}
Given metric spaces $(X,d_X)$ and $(Y,d_Y)$, a map $f:X \to Y$ is said to be {\em an $L$-bi-Lipschitz embedding} (or simply {\em $L$-bi-Lipschitz} or just {\em bi-Lipschitz}) if there is a constant $L \geq 1$ such that
$$
L^{-1} d_X(x_1,x_2) \leq d_Y(f(x_1),f(x_2)) \leq L d_X(x_1,x_2)
$$
for all
$x_1,x_2 \in X$.
A \emph{bi-Lipschitz arc} in a metric space $X$
is the image of an interval in the real line $\R$ under a bi-Lipschitz map.

We will consider the following question:
given a set $E \subset X$ which is the image of a subset of $\R$ under a bi-Lipschitz map,
is $E$ contained in a bi-Lipschitz arc?
If $E$ is any finite subset of $\R^n$, the answer is trivially ``yes''.
For general sets $E \subset \R^n$,
the question was answered in the positive when $n \geq 3$
by the following extension theorem of David and Semmes \cite{David:1991}:
\begin{thm}[{\cite[Proposition 17.1]{David:1991}}]\label{thm:DS}
Let $n\geq 3$ be an integer, let $A \subset \R$, and let $f:A \to \R^n$ be a bi-Lipschitz embedding.
Then there exists a bi-Lipschitz extension $F:\R \to \R^n$.
\end{thm} 
MacManus \cite{MacManus:1995} extended the result of David and Semmes to the case $n=2$, which is much more difficult since intersecting lines in $\R^3$ may be easily modified so that they no longer intersect, but this is not the case in $\R^2$.
One may view these extension results as rougher versions of the classical Whitney Extension Theorem \cite{W34}; while the maps considered here are analytically weaker (as they are bi-Lipschitz rather than differentiable), they are metrically and topologically stronger.

Theorem \ref{thm:DS} is a special case of a more general result in \cite{David:1991} where $A\subset \R^d$ and $n\geq 2d+1$.
The main motivation behind that result was to establish the equivalence of the 
boundedness of certain singular operators on $\R^n$ via quantitative rectifiability.
More precisely, Theorem \ref{thm:DS} was used in \cite{David:1991} to show that, when  $n \geq 3$, every Ahlfors $1$-regular set $A \subset \R^n$ (see \eqref{eq:reg} for the definition of Ahlfors regularity) which admits a corona decomposition
(roughly speaking, $A$ can be decomposed into a collection of subsets which are well-approximated by Lipschitz graphs and a collection of subsets which are not, and both of these collections have controlled measure)
contains ``big pieces'' of bi-Lipschitz arcs
i.e. for any $\e>0$, there exists an $M>0$ such that, 
for any $x\in A$ and any $R>0$, there is an $M$-bi-Lipschitz embedding $\rho:\R \to \R^n$ such that
\[ |E \cap (B(x,R)\setminus \rho(\R))| \leq \e R.\]

Another application of Theorem \ref{thm:DS} 
is in the problem of the 
\emph{bi-Lipschitz rectifiability} of sets in Euclidean spaces.
In other words, one hopes to classify those subsets of $\R^n$ that are contained in a bi-Lipschitz arc.
While the classical characterization of the Lipschitz rectifiability of sets in Euclidean spaces has been completely resolved \cite{Jones90, Ok92}, 
the problem of bi-Lipschitz rectifiability remains open mainly due to topological constraints.
Theorem \ref{thm:DS} can be used to show that, if a set $E \subset \R^n$ has Assouad dimension less than 1,
then $E$ is bi-Lipschitz rectifiable; see \cite[Corollary 3.5]{Badger:2019} for a different approach.
See Section~\ref{sec:prelim} for the definition of the Assouad dimension.

In this article, we generalize Theorem~\ref{thm:DS} to the setting in which Euclidean spaces $\R^n$ are replaced by a large class of metric measure spaces.
There are two main difficulties in this generalization.
Firstly, the target metric space $X$ must contain many of rectifiable curves, and this notion of ``many'' must be understood quantitatively.
A notable example (and, in fact, the initial motivation for this project) is the \emph{Heisenberg group $\mathbb{H}$}
in which the classical Whitney Extension Theorem for curves has been well-studied recently; see \cite{Z18, PSZ19, Z22, SZ22}.
We will not define the Heisenberg group here but only recall that it is a geodesic space homeomorphic to $\R^3$, and there exists a distribution $H:\R^3 \to \textbf{Gr}(2,\R^3)$ such that if a curve $\gamma:[0,1] \to \mathbb{H}$ is rectifiable, then it is differentiable almost everywhere and $\dot{\gamma}(t) \in H_{\gamma(t)}$ for almost every $t$.
This fact implies that there must be many fewer rectifiable curves in $\mathbb{H}$ than in $\R^3$.
Secondly, the proof in the Euclidean case relies on the existence of differentiable bump functions $\phi:\R\to\R^n$ with controlled derivatives, and we cannot hope to recover this idea in a general metric space.

The class of metric measure spaces to which the bi-Lipschitz extension result will be generalized will have two properties.
The first is \emph{Ahlfors regularity}: we say that a metric measure space $(X,d,\mu)$ is Ahlfors $Q$-regular (or simply $Q$-regular) if the measure of any ball of radius $r$ is comparable to $r^Q$. The second property is the existence of a \emph{Poincar\'e inequality}. Such an inequality roughly states that, if we use $u_B$ to denote the average value of a function $u:X \to \R$ on a ball $B$, then the average of the variation $|u-u_B|$ is controlled by the average of a ``weak derivative'' of $u$ on $B$.
See Section \ref{sec:prelim} for all relevant definitions.
It is known that Ahlfors regular spaces supporting a Poincar\'e inequality must contain quantitatively many rectifiable curves. Moreover, such spaces admit a notion of differentiation \cite{Cheeger99}.

The following is the main result of this paper:
\begin{thm}\label{thm:1}
Let $(X,d,\mu)$ be a $Q$-regular, complete metric measure space supporting a $p$-Poincar\'e inequality for some $1< p \leq Q-1$. If $A \subset \R$ and $f:A \to X$ is a bi-Lipschitz embedding, then $f$ extends to a bi-Lipschitz embedding $F:I \to X$ where $I$ is the smallest closed interval containing $A$.
\end{thm}

In Theorem~\ref{thm:main2} we prove a stronger quantitative version of this result in the sense that
the bi-Lipschitz constant of $F$ depends only the bi-Lipschitz constant of $f$ and on the data of Ahlfors $Q$-regularity and the Poincar\'e inequality.
Moreover, if $X$ is unbounded, then we can choose $I=\R$.

A large variety of metric spaces satisfy the assumptions of Theorem~\ref{thm:1}, including 
orientable, $n$-regular, linearly locally contractible $n$-manifolds with $n\geq 3$ \cite{Semmes96}, Carnot groups \cite{Varopoulos, Jerison86} (which include Euclidean spaces and the Heisenberg group), certain hyperbolic buildings \cite{BP99}, Laakso spaces \cite{Laakso00}, and certain Menger sponges \cite{MTW,EBG}. 

The assumptions of the theorem are sharp in that neither Ahlfors regularity nor the Poincar\'e inequality can be removed from the statement. For Ahlfors regularity, let $X=\S^2\times \R$ with the length metric and the induced Hausdorff 3-measure. Then $X$ is complete, has Ricci curvature bounded from below so it satisfies the $1$-Poincar\'e inequality \cite[Chapter VI.5]{Chavel}, but is not Ahlfors regular. Define $f:\{2^{n}:n\in\N\} \to X$ by $f(2^n) = (p_0,(-2)^n)$ where $p_0 \in \S^2$. The map $f$ is bi-Lipschitz and if $F:\R \to X$ is any homeomorphic extension of $f$, then for any $n\in\N$, $F([2^n,2^{n+1}])$ intersects with $(\S^2\times\{0\})$ 
so $F$ can not be bi-Lipschitz.

Since the Poincar\'e inequality is an open ended condition \cite{KZ08}, we may assume that $p<Q-1$ for the proof of the theorem. However, the bound $Q-1$ is sharp. To see this, let $n\geq 2$, let $P_1,P_2$ be two $n$-dimensional planes in $\R^{2n-1}$ intersecting on a line $\ell$, and let $p_0\in \ell$. The metric space $X = (P_1\cup P_2)\setminus B(p_0,1)$ with the induced Euclidean metric and $n$-dimensional Lebesgue measure is complete, $n$-regular, and satisfies the $p$-Poincar\'e inequality for all $p>n-1$ \cite[Theorem 6.15]{HK}. Let $f: (-\infty,-1]\cup \{-\frac12,\frac12\} \cup [1,\infty) \to X$ be a map such that $f(-\tfrac12) \in P_1\setminus (\ell \cup B(p_0,1))$, $f(\tfrac12) \in P_2\setminus (\ell\cup B(p_0,1))$, and $f$ maps $\R\setminus (-1,1)$ isometrically onto $\ell\setminus B(p_0,1)$. Then $f$ is bi-Lipschitz but admits no homeomorphic (let alone bi-Lipschitz) extension $F:\R\to X$.


\subsection{Related results}

The first corollary of Theorem \ref{thm:1} gives a sufficient condition for bi-Lipschitz rectifiability in Ahlfors regular spaces satisfying a Poincar\'e inequality.

\begin{cor}
Let $X$ be a complete $Q$-regular metric measure space supporting a $p$-Poincar\'e inequality for some $1< p \leq Q - 1$. If $E \subset X$ has Assouad dimension less than 1, then $E$ is bi-Lipschitz rectifiable.
\end{cor}

The proof of the corollary follows the same ideas as in the Euclidean case. Since the Assouad dimension of $E$ is less than 1, \cite[Lemma 15.2]{DS97} implies that $E$ must be uniformly disconnected, and hence it is bi-Lipschitz equivalent to an ultrametric space $Z$ of Assouad dimension less than 1 \cite[Proposition 15.7]{DS97}.
By \cite[Theorem 3.8]{Luukkainen:1994}, there exists a bi-Lipschitz embedding $g:E \to \R$, and, by Theorem~\ref{thm:1}, there exists a closed interval $I$ and a bi-Lipschitz extension $f:I \to X$ of the map $g^{-1} : g(E) \to X$.
Thus $E \subset f(I)$, so $E$ is contained in a bi-Lipschitz arc.


The proof of Theorem~\ref{thm:1} has two main ingredients. The first is the construction of short curves in $X \setminus f(A)$ that stay quantitatively far from $f(A)$.
To build such curves, we will use the notion of the uniformity of a set.
Given a set $U \subset X$, we say that $U$ is $c$-\emph{uniform} if, for every $x,y \in U$, there exists a path $\gamma: [0,1] \to U$ joining $x$ to $y$ such that
\begin{enumerate}
\item the length of $\gamma$ is at most $c d(x,y)$, and
\item $\dist(\gamma(t), X \setminus U) \geq c^{-1}\dist(\gamma(t), \{x,y\})$ for all $t\in [0,1]$.
\end{enumerate} 
In other words, $U$ is uniform if, for any $x,y \in U$, there exists a curve connecting them which is short compared to $d(x,y)$ and stays far from $X\setminus U$ quantitatively. If $U$ satisfies only property (1) in this definition, then we say that $U$ is $c$-\emph{quasiconvex}.

It is an open problem to classify the closed sets $Y\subset X$ for which $X\setminus Y$ is quasiconvex or uniform.
Hakobyan and Herron \cite{HH08} showed that, if $Y\subset \R^n$ has Hausdorff $(n-1)$-measure $\mathcal{H}^{n-1}(Y) =0$, then $\R^n\setminus Y$ is quasiconvex.
Moreover, this assumption is sharp. 
Herron, Lukyanenko, and Tyson \cite{HLT} proved the same result in the Heisenberg group $\mathbb{H}$ where, in this setting, it is assumed that $\mathcal{H}^{3}(Y) =0$.
The dimension $3$ is natural as $\mathbb{H}$ is 4-regular while $\R^n$ is $n$-regular.
It is unknown if a similar result exists in all Carnot groups.

The question of whether $X\setminus Y$ is uniform has been studied in terms of uniform disconnectedness of $Y$ \cite{MacManus:1999} and quasihyperbolicity of $X$ and $Y$ \cite{Herron1, HVW17, Herron2}. V\"ais\"al\"a \cite{Vais88} showed that, 
if $\R^n\setminus Y$ is uniform, then the topological dimension of $Y$ is at most $n-2$. The following proposition, which we prove in Section \ref{sec:uniform}, works in the opposite direction: if $X$ is Ahlfors regular and supports a Poincar\'e inequality and if the Assouad dimension of $Y$ is small, then $X\setminus Y$ is uniform. 

\begin{prop}\label{Prop:Unif}
Let $(X,d,\mu)$ be a complete $Q$-Ahlfors regular metric measure space supporting a $p$-Poincar\'e inequality for some $1< p \leq Q$. If $Y \subset X$ is a closed set with Assouad dimension less than $Q - p$, then $X \setminus Y$ is a uniform domain.
\end{prop}

Note that if $Y \subset X$ and has Assouad dimension less than $Q-p$, then $\mathcal{H}^{Q-p}(Y)=0$. The assumption on the Assouad dimension is sharp. For example, let $X=\R^n$, let $P$ be an $(n - 1)$-dimensional hyperplane in $\R^n$, and let $Y$ be a maximal $1$-separated subset of $P$. Then it is easy to see that $\dim_A(Y) = n - 1$, $\mathcal{H}^{n-1}(Y)=0$, and $\R^n \setminus Y$ is not uniform.

The second ingredient in the proof of Theorem \ref{thm:1} is a standard ``straightening" argument for paths. In particular, Lytchak and Wenger \cite[Lemma 4.2]{LW20} proved that, given any topological arc in a geodesic space, there exists a bi-Lipschitz arc with the same endpoints that is close to the original one; see also \cite[Lemma 4.2]{MW21} for a similar result for topological circles. In Section \ref{sec:approx}, we prove a quantitative version of their result.
Moreover, under the additional assumptions of $Q$-regularity and a Poincar\'e inequality, we show as a corollary of Theorem \ref{thm:1} that every continuum (i.e., every compact connected set) can be approximated by a bi-Lipschitz curve in the Hausdorff distance.

\begin{prop}\label{cor:Modi}
Let $(X,d,\mu)$ be a complete $Q$-regular metric measure space supporting a $p$-Poincar\'e inequality for some $1< p \leq Q - 1$, let $K \subset X$ be a continuum, and let $\e \in (0,1)$.
For any $x,y\in K$ with $d(x, y) \geq \e \diam{K}$, there exists a curve $\gamma:[0,1] \to X$ with $\gamma(0)=x$
and $\gamma(1)=y$,
and there exists a constant $L\geq 1$ depending only on $\e$, the constants of $Q$-regularity, and the data of the Poincar\'e inequality such that
\[ \frac1{L}|s-t|\diam{K} \leq d(\gamma(t),\gamma(s)) \leq L|s-t|\diam{K}\]
for all $s,t \in [0,1]$,
and the Hausdorff distance $\dist_H(K,\gamma([0,1])) \leq \e \diam{K}$. 
\end{prop}

In particular, every compact Ahlfors regular metric measure space supporting a Poincar\'e inequality contains ``almost space-filling'' bi-Lipschitz curves.


\subsection{Outline of the proof of Theorem \ref{thm:1}}
We start with two simple reductions. First, since bi-Lipschitz maps extend on the completion of their domain, we may assume that $A$ is a closed set. 
Second, it is well known that the Poincar\'e inequality, completeness, and Ahlfors regularity imply that $X$ is quasiconvex \cite[Theorem 17.1]{Cheeger99}. Every complete quasiconvex space is bi-Lipschitz equivalent to a geodesic metric space and since the properties of Ahlfors $Q$-regularity and the $p$-Poincar\'e inequality are preserved under bi-Lipschitz mappings \cite[Lemma 8.3.18]{HKST}, we may assume for the rest that $X$ is geodesic.

For the proof of Theorem \ref{thm:1}, similar to the proof of Theorem \ref{thm:DS} and the Whitney Extension Theorem, we construct a \emph{Whitney decomposition} $\{\mathcal{Q}_i\}_{i\in\N}$ of $I \setminus A$, i.e., a collection of closed intervals in $I\setminus A$ with mutually disjoint interiors such that their union is $I \setminus A$ and the length of each interval is comparable to its distance from $A$. 

In Section \ref{sec:prelimembeddings} we define two auxiliary embeddings. Specifically, in \textsection\ref{sec:endpoints} we construct a bi-Lipschitz embedding $\pi$ of $E$ into $X$, where $E$ is the set of endpoints of the Whitney intervals $\mathcal{Q}_i$. The final map $F$ will map elements of $E$ very close to their image under $\pi$. In \textsection\ref{sec:middlethird}, we use the results of Sections~\ref{sec:uniform} and \ref{sec:approx} to define a second bi-Lipschitz embedding 
\[ g: A\cup\bigcup_{i\in\N}\hat{\mathcal{Q}}_i \to X\]
of $f$. Here, $\hat{\mathcal{Q}}_i$ denotes the middle third closed interval in $\mathcal{Q}_i$. If we write $\mathcal{Q}_i = [x,y]$, then the image $g(\hat{\mathcal{Q}}_i)$ is a bi-Lipschitz curve that has endpoints very close to $\pi(x)$ and $\pi(y)$.

In Section~\ref{sec:proof} we describe a method to modify and extend the map $g$ near the points $\pi(x)$ to build a curve on the entire interval $I$, and we verify that this curve is indeed bi-Lipschitz to complete the proof of Theorem \ref{thm:1}.

\subsection*{Acknowledgements} We thank the referee for their valuable comments. The second author would like to thank Sylvester Eriksson-Bique for a valuable conversation at an early stage of this project, and Damaris Meier for a conversation on the bi-Lipschitz approximation of curves.

\section{Preliminaries}\label{sec:prelim}

Given quantities $x,y \geq 0$ and constants $a_1,\dots,a_n>0$ we write $x \lesssim_{a_1,\dots,a_n} y$ if there exists a constant $C$ depending at most on $a_1,\dots,a_n$ such that $x \leq C y$. If $C$ is universal, we write $x\lesssim y$. We write $x\simeq_{a_1,\dots,a_n} y$ if $x\lesssim_{a_1,\dots,a_n}y$ and $y\lesssim_{a_1,\dots,a_n} x$. 

Given a metric space $(X,d)$ and two points $x,y \in X$, we say that $\gamma$ is a {\em path joining $x$ with $y$} if there exists some continuous $\gamma:[0,1] \to X$ with $\gamma(0) = x$ and $\gamma(1) = y$.

Given a set $Y \subset X$ and $r>0$, we write
$
B(Y,r) := \{ x \in X : \dist(x,Y) < r\}.
$

\subsection{Porosity and regularity}

For a constant $C>1$, a metric space $X$ is called \emph{$C$-doubling} if every ball of radius $r$ can be covered by at most $C$ balls of radii at most $r/2$. 
Given another constant $\alpha \geq 0$, $X$ is called \emph{$(C,\alpha)$-homogeneous} if every ball of radius $R$  can be covered by at most $C(R/r)^{\alpha}$ balls of radii at most $r$. 
We will occasionally refer to such a metric space as $\alpha$-{\em homogeneous} when the constant $C$ is not important.
Clearly, a $(C,\alpha)$-homogeneous space is $(C2^{\alpha})$-doubling. Conversely, given $C>0$ there exists $C'>0$ and $\alpha>0$ such that a $C$-doubling space is $(C',\alpha)$-homogeneous.

The \emph{Assouad dimension} of a metric space $X$ (denoted $\dim_A(X)$) is the infimum of all $\alpha \ge 0$ such that $X$ is $\alpha$-homogeneous.

A metric measure space $(X,d,\mu)$ is said to be \emph{$Q$-Ahlfors regular} (or $Q$-{\em regular}) for $Q \ge 0$ if there exists $C \ge 1$ such that, for all $x \in X$ and all $r \in (0,\diam{X})$,
\begin{equation}\label{eq:reg} 
C^{-1}r^Q \le \mu(B(x,r)) \le Cr^Q. 
\end{equation}
It is easy to see that if $(X,d,\mu)$ is $Q$-regular, then $X$ is $Q$-homogeneous and $\dim_A(X) = Q$.
If we want to emphasize the constant $C$ in \eqref{eq:reg}, then we say that $(X,d,\mu)$ is $(C,Q)$-{\em regular}.

Given $Y \subset X$ we say that $Y$ is \emph{$p$-porous} for some $p\geq 1$ if, for all $y\in Y$ and all $r \in (0,\diam{X})$, there exists some $x \in B(y,r)$ such that $B(x,r/p) \subset B(y,r) \setminus Y$.
In other words, $Y$ contains relatively large ``holes'' near every point.

\begin{lem}[{\cite[Lemma 3.12]{Bonk:2001}}]\label{lem:por}
Let $(X,d,\cH^Q)$ be $(C,Q)$-regular, where $\cH^Q$ is the $Q$-dimensional Hausdorff measure. A set $Y\subset X$ is $p$-porous for some $p\geq 1$ if and only if $\dim_A(Y) \leq Q-\e$ for some $\e>0$. Here, $\e$ and $p$ depend only on each other, $Q$, and $C$.
\end{lem}

\subsection{Poincar\'e inequality}\label{subsec:Poincare}

Given a locally Lipschitz function $u$ defined on a metric space $(X,d)$, we say that a function $g:X \to [0,\infty)$ is an \emph{upper gradient} of $u$ if
\[ |u(x) - u(y)| \le \int_{\gamma} g\,\d s \]
for all $x,y \in X$ and all paths $\gamma$ in $X$ joining $x$ with $y$.
	
We say that a metric measure space $(X,d,\mu)$ supports a \emph{$(1,p)$-Poincar\'e inequality} (or simply a \emph{$p$-Poincar\'e inequality}) for some $1 \leq p < \infty$ if there exist $\lambda \geq 1$ and $C > 1$ with the following property: if $u:X \to \R$ is locally Lipschitz and $g:X \to [0,\infty)$ is an upper gradient of $u$, then, for all $x \in X$ and $r > 0$,
\begin{equation}\label{eq:PI} 
\fint_{B(x,r)} |u - u_{B(x,r)}|\,\d\mu \le C\diam(B(x,\lambda r))\left(\fint_{B(x,\lambda r)} g^p\,\d\mu\right)^{1/p}, 
\end{equation}
where
\[ \fint_A f\,\d\mu = \frac{1}{\mu(A)}\int_A f\,\d\mu \]
and
\[ u_{B(x,r)} = \fint_{B(x,r)} u\,\d\mu. \]

It follows from H\"older's inequality that if $1\leq p \leq q$ and $(X,d,\mu)$ satisfies a $p$-Poincar\'e inequality, then it satisfies a $q$-Poincar\'e inequality. Moreover, if the space is geodesic and doubling, then one can choose $\lambda=1$; see for example \cite[Remark 9.1.19]{HKST}. Henceforth, given a geodesic doubling space $X$ that satisfies the $p$-Poincar\'e inequality, we will assume that $\lambda=1$ in \eqref{eq:PI} and the constant $C$ will be called the \emph{data} of the Poincar\'e inequality.

For a detailed exposition on the Poincar\'e inequality on metric measure spaces, the reader is referred to \cite{HKST}.

\subsection{Modulus of curve families}

The basic tool in the proof of Theorem \ref{thm:1} and Proposition \ref{Prop:Unif} is the notion of the modulus of curves. 
In a sense, the modulus is a measurement of ``how many'' rectifiable curves are contained in a curve family.

Given a family $\Gamma$ of rectifiable curves in a metric measure space $(X,d,\mu)$, we say that a Borel function $\rho:X \to [0,\infty)$ is admissible for $\Gamma$ if
\[ \int_{\gamma} \rho\,\d s \ge 1 \quad \text{ for all } \gamma \in \Gamma. \]
For $p \ge 1$, we define the \emph{$p$-modulus} of $\Gamma$ by
\[ \Mod_p(\Gamma) := \inf\left\{\int_X \rho^p\,\d\mu : \rho \text{ is admissible for } \Gamma\right\}. \]
It is well known that $\Mod_p$ is an outer measure on the space of all curve families in $X$.

The next lemma relates the modulus of curve families with the \emph{locally Lipschitz capacity} between compact sets.
Given two sets $E$ and $F$ in a metric space, we say that a curve $\gamma$ joins $E$ with $F$ if there are points $x \in E$ and $y \in F$ such that $\gamma$ joins $x$ with $y$.

\begin{lem}[{\cite[Theorem 1.1]{Kallunki:2001}}]\label{Lem:Cap}
Suppose that $(X,d,\mu)$ is a geodesic metric measure space equipped with a doubling measure $\mu$ and supporting a $p$-Poincar\'e inequality with $p>1$, and suppose that $\Omega$ is a domain in $X$. Let $E,F$ be disjoint, compact, non-empty subsets of $\Omega$, and let $\Gamma$ be the collection of curves in $\Omega$ that join $E$ with $F$. Then the $p$-modulus of $\Gamma$ is equal to the $p$-capacity of $E$ and $F$:
\[ \Mod_p(\Gamma) = \Capa_p(E,F) := \inf \int_{\Omega} g^p\,\d\mu, \]
where the infimum is taken over all 
Borel functions $g:\Omega \to [0,\infty)$ such that each $g$ is an upper gradient of some
locally Lipschitz function $u:\Omega \to \R$ satisfying $u|_E \ge 1$ and $u|_F \le 0$.
\end{lem}

\section{Uniformity in metric measure spaces}\label{sec:uniform}

The goal of this section is the proof of Proposition \ref{Prop:Unif}. The next lemmas are the crux of the proof.

\begin{lem}\label{Lem:Path1}
Let $(X,d,\mu)$ be a $(C_1,Q)$-Ahlfors regular geodesic metric measure space supporting a $p$-Poincar\'e inequality with data $C$, for some $C,C_1 \geq 1$, and $1<p<Q$. Let $x,y \in X$, let $r\in (0,\frac13d(x,y))$, and let $\Gamma$ be the collection of paths in $B(x,2d(x,y))$ that connect $\overline{B}(x,r)$ with $\overline{B}(y,r)$. Then
\[ \Mod_p\left(\Gamma\right) \gsim_{p,C,C_1,Q} (d(x,y))^{Q-p}\left(\frac{d(x,y)}{r}\right)^{-Qp}.\]
\end{lem}

\begin{proof}
Set $D = d(x,y)$. Let $u:B(x,2D) \to \R$ be a locally Lipschitz function satisfying $u|_{B(x,r)} \ge 1$ and $u|_{B(y,r)} \le 0$.
Let also $g:B(x,2D) \to [0,\infty)$ be an upper gradient of $u$. By the $p$-Poincar\'e inequality,
\begin{align*}
&\int_{B(x,2D)} g^p\,\d\mu \\
& \ge \frac{\mu(B(x,2D))}{C^p[\diam(B(x,2D))\mu(B(x,2D))]^p}\left(\int_{B(x,2D)} |u - u_{B(x,2D)}|\,\d\mu \right)^p \\
&\gtrsim_{p,C,C_1,Q}D^{Q - p - Qp}\left(\int_{B(x,r) \cup B(y,r)} |u - u_{B(x,2D)}|\,\d\mu\right)^p \\
& \ge D^{Q - p - Qp}\left(\tfrac12\min\{\mu(B(x,r)),\mu(B(y,r))\}\right)^p \\
& \gsim_{p,C_1} D^{Q-p}\left(\frac{D}{r}\right)^{-Qp}. \\
\end{align*}

Denote by $\Gamma$ the collection of curves joining $\overline{B}(x,r)$ with $\overline{B}(y,r)$. 
By Lemma \ref{Lem:Cap},
\[ \Mod_p(\Gamma) \gtrsim_{p,C,C_1,Q} D^{Q-p}\left(\frac{D}{r}\right)^{-Qp}.
\qedhere\]
\end{proof}

\begin{lem}\label{Lem:Path2}
Let $(X,d,\mu)$ be a $(C_1,Q)$-Ahlfors regular metric measure space, let $R>0$, let $\ell>0$, and let $\Gamma$ be the collection of paths in $B(x,R)$ that have length at least $\ell R$. Then,
\[ \Mod_p\left(\Gamma\right) \lesssim_{C_1} \ell^{-p}R^{Q-p}.\]
\end{lem}

\begin{proof}
Note that the function $\rho = (\ell R)^{-1}\chi_{B(x,R)}$ is admissible for $\Gamma$. Therefore
\[ \Mod_p(\Gamma) \le \int_X \rho^p\,\d\mu \le C_1\ell^{-p}R^{Q-p}. \qedhere \]
\end{proof}

\begin{lem}\label{Lem:Path3}
Let $(X,d,\mu)$ be a $(C_1,Q)$-Ahlfors regular metric measure space, let $Y \subset X$ be a $(C_2,\alpha)$-homogeneous set, let $R>0$, let $\delta >0$, and let $\Gamma$ be the collection of paths in $B(x,R)$ with an endpoint outside of $B(Y,2\delta R)$ and which intersect $B(Y,\delta R)$. Then
\[ \Mod_p\left(\Gamma\right) \lesssim_{Q,C_1,C_2} \delta^{Q-p-\alpha}R^{Q-p}. \]
\end{lem}

\begin{proof}
Define the function
\[ \rho := (\delta R)^{-1}\chi_{B(Y,2\delta R) \cap B(x,R)} \]
and note that $\rho$ is admissible for $\Gamma$. Indeed, if $\gamma \in \Gamma$, then the the total length of the part of $\gamma$ that is inside $B(Y,2\delta R)$ must be at least $\delta R$. 

If $V$ is a $(\delta R)$-net of $Y \cap B(x,R)$, then
\[ B(Y,2\delta R) \cap B\left(x,R\right) \subset \bigcup_{v \in V} B(v,3\delta R), \]
and, by the homogeneity of $X$, it follows that $\card(V) \lesssim_{C_2} \delta^{-\alpha}$. Therefore
\[ \Mod_p(\Gamma) \le \int_X \rho^p\,\d\mu \lesssim_{Q,C_1,C_2} \delta^{Q-p-\alpha} R^{Q - p}. \qedhere\]
\end{proof}

\begin{cor}\label{Cor:Path}
Let $(X,d,\mu)$ be a $(C_1,Q)$-regular, geodesic metric measure space supporting a $p$-Poincar\'e inequality with $1<p<Q$ and data $C$. Let $Y\subset X$ be a $(C_2,\alpha)$-homogeneous set with $0<\alpha<Q-p$. Given $x,y \in X\setminus Y$, there exists a path $\gamma:[0,1] \to X\setminus Y$ such that $\gamma(0)=x$, $\gamma(1)=y$,
\begin{enumerate}
\item $\gamma([0,1]) \subset B(x,2 d(x,y))$,
\item
\[ \len(\gamma) \lesssim_{p,C,C_1,Q} d(x,y)\max\left\{1,\left(\frac{d(x,y)}{\dist(\{x,y\},Y)}\right)^Q\right\}\]
\item for all $z$ in the image of $\gamma$,
\[ \dist(z,Y) \gtrsim_{p,\alpha,Q,C,C_1,C_2} d(x,y)\min\left\{1,\left(\frac{\dist(\{x,y\},Y)}{d(x,y)}\right)^{\frac{Qp + Q - p - \alpha}{Q - p - \alpha}}\right\}. \]
\end{enumerate}
\end{cor}

\begin{proof}
Set $D := d(x,y)$ and
\[ r := \tfrac14\min\left\{D,\dist(\{x,y\},Y)\right\}.\]
Let $\Gamma_1$ be the collection of all curves in $B(x,2D)$ that join $\overline{B}(x,r)$ to $\overline{B}(y,r)$. Let $\Gamma_{\ell}$ be the collection of all curves in $B(x,2D)$ that have length at least $2D\ell$. Let $\Gamma'_{\delta}$ be the collection of all curves in $B(x,2D)$ that intersect a $(2D \delta)$-neighborhood of $Y$ and have length at least $2D \delta$.

By Lemma \ref{Lem:Path1}, Lemma \ref{Lem:Path2}, and Lemma \ref{Lem:Path3}, there exist 
\begin{align*}
\ell &\simeq_{p,C,C_1,Q}\left(\frac{D}{r}\right)^Q \quad\text{and}\quad \delta \simeq_{p,\alpha,Q,C,C_1,C_2} \left( \frac{r}{D}\right)^{\frac{Qp}{Q-p-\alpha}}
\end{align*}
such that 
\[ \Mod_p(\Gamma \setminus (\Gamma_l \cup \Gamma_{\delta}')) >0.\]
It follows that $\Gamma \setminus (\Gamma_l \cup \Gamma_{\delta}')$ is non-empty. Fix now $\gamma \in \Gamma \setminus (\Gamma_l \cup \Gamma_{\delta}')$ and concatenate $\gamma$ with geodesic segments $[x,\gamma(0)]$ and $[\gamma(1),y]$. The resulting curve satisfies the conclusions of the corollary.
\end{proof}

\begin{proof}[Proof of Proposition \ref{Prop:Unif}]
By Lemma \ref{lem:por}, the regularity of $X$, and the homogeneity of $Y$, there exists $p_0 > 1$ such that $Y$ is $p_0$-porous.

Fix now $x,y \in X \setminus Y$ and denote $r := d(x,y)$.
There exists $z_0 \in B(x,r) \setminus (B(x,2^{-1}r) \cup B(y,2^{-1}r))$ such that
\[ B(z_0,2^{-1}r/p_0) \subset B(x,r) \setminus (B(x,2^{-1}r) \cup B(y,2^{-1}r) \cup Y) \]
by applying porosity of $Y$ to a ball of radius $2^{-1}r$ contained in
\[ B(x,r) \setminus (B(x,2^{-1}r) \cup B(y,2^{-1}r)). \]
Moreover, for each $n \in \N$, there exist points $z_n \in B(x,2^{-n}r) \setminus B(x,2^{-n - 1}r)$ and $z_{-n} \in B(y,2^{-n}r) \setminus B(y,2^{-n - 1}r)$ such that
\[B(z_n,2^{-n - 1}r/p_0) \subset B(x,2^{-n}r) \setminus (B(x,2^{-n - 1}r) \cup Y)\] 
and 
\[B(z_{-n},2^{-n - 1}r/p_0) \subset B(y,2^{-n}r) \setminus (B(y,2^{-n - 1}r) \cup Y),\]
again by applying the porosity of $Y$ to balls in the annuli $B(x,2^{-n}r) \setminus B(x,2^{-n - 1}r)$ and $B(y,2^{-n}r) \setminus B(y,2^{-n - 1}r)$.

Applying Corollary \ref{Cor:Path}, there exists $c > 1$ depending only on $p_0$, $p$, $Q$, $C$, and $C_1$ such that, for each $n \in \Z$, there exists a path $\gamma_n:[0,1] \to X \setminus Y$ with
\begin{enumerate}
\item $\gamma_n(0) = z_n$, $\gamma_n(1) = z_{n + 1}$,
\item $\text{length}(\gamma_n) \le cd(z_n,z_{n + 1}) \le 2^{3 - |n|}cr$, and
\item for all $t \in [0,1]$, $\dist(\gamma_n(t),Y) \ge c^{-1}2^{-2 - |n|}r$.
\end{enumerate}

Concatenating all the paths $\{\gamma_n\}_{n \in \Z}$ and adding the points $x,y$ we obtain a path $\gamma:[0,1] \to X \setminus Y$. Note that
\[ \text{length}(\gamma) = \sum_{n \in \Z} \text{length}(\gamma_n) \le \sum_{n \in \Z} 2^{3 - |n|}cr = 24cr = 24cd(x,y). \]
Let now $z \in \gamma([0,1])$.
If $z$ is either of $x$ or $y$, then there is nothing to show.
Otherwise, there exists $n \in \Z$ such that $z$ is in the image of $\gamma_n$. Assume as we may that $n \ge 0$. Then
\[ d(x,z) \le d(z_n,x) + d(z_n,z) \le (8c + 1)2^{-n}r \le 4c(8c + 1)\dist(z,Y), \]
which completes the proof.
\end{proof}

\section{Bi-Lipschitz approximation of curves}\label{sec:approx}

In this section we show how paths in geodesic spaces can be approximated by bi-Lipschitz arcs with the same endpoints. The main goal will be the proof of Proposition~\ref{cor:Modi}.

The next lemma is important in the proof of Theorem \ref{thm:1} and is almost identical to \cite[Lemma 4.2]{LW20}. The difference here is the quantitative control on the bi-Lipschitz constant $L$.

\begin{lem}\label{Prop:Modi}
Given $C\geq 1$ and $\e>0$, there exist $L=L(C,\e)\ \geq 1$ with the following property. Let $(X,d)$ be a $C$-doubling geodesic metric space and let $\sigma:[0,1] \to X$ be a curve with $\sigma(0) \neq \sigma(1)$. 
There exists a curve 
$\gamma:[0,1] \to X$ such that $\gamma(0) = \sigma(0)$, $\gamma(1) = \sigma(1)$, for all $s,t\in [0,1]$
\[\frac1{L}|s-t|\diam{\sigma([0,1])} \leq d(\gamma(s),\gamma(t)) \leq L|s-t|\diam{\sigma([0,1])}, \]
and
\[ \dist(\gamma(t),\sigma([0,1])) \leq \e \diam{\sigma([0,1])}.\]
\end{lem}

The doubling property is not necessary to guarantee the existence of the bi-Lipschitz map $\gamma$; see \cite[Lemma 4.2]{LW20}. It is, however, necessary to control the constant $L$. For example, let $X=\ell_2$, let $e_1,e_2,\dots$ be an orthonormal basis of $\ell_2$, and let $n\in\N$. Define $\sigma:[0,1] \to \ell_2$ so that $\sigma(0)=e_0 :=0$, $\sigma(i/n)=e_i$ for $i\in\{1,\dots,n\}$, and $\sigma|_{[(i-1)/n,i/n]}$ is linear for each $i\in \{1,\dots,n\}$. Note that $\diam\sigma([0,1]) = \sqrt2$. It is easy to see that, if $\e < 1/6$, then for each $i\in \{1,\dots,n-1\}$ the set 
\[ 
B(\sigma([0,1]),\sqrt2\e) 
\setminus B(\sigma(i/n),3\sqrt2\e)
\]
is disconnected. Therefore, if $\gamma$
is a 
path in $\ell_2$
joining $0$ with $e_n$ and
satisfying $\gamma([0,1]) \subset B(\sigma([0,1]),\sqrt2\e)$, then $\gamma([0,1])$ must intersect each ball $B\left(\sigma(i/n),3\sqrt2\e\right)$ for $i = 1,\ldots,n$. In particular, the length of $\gamma$ is at least a fixed multiple of $n$, while $|\gamma(0)-\gamma(1)| = 1$. It follows that, if $\gamma$ is $L$-bi-Lipschitz, then $L$ must depend on $n$ and not just on $\e$.

For the proof of Lemma~\ref{Prop:Modi}, we require a simple lemma. 
Here and for the rest of this section, all geodesic curves are parameterized by arc-length.

\begin{lem}\label{Lem:BL}
Let $X$ be a geodesic metric space, let $a\geq b>0$, let $f:[0,a] \to X$ be $L$-bi-Lipschitz, let $p \in X$, and suppose $f(b)$ is the closest point in $f([0,a])$ to $p$, i.e.,
\[ c:= \dist(f([0,a]),p) = d(f(b),p).\]
If $g:[b,b+c] \to X$ is the geodesic from $f(b)$ to $p$, then the concatenation of $f|_{[0,b]}$ and $g$ is $(2L)$-bi-Lipschitz.
\end{lem}

\begin{proof}
Let $h:[0,b+c] \to X$ be the concatenation of $f|_{[0,b]}$ and $g$. Clearly $h|_{[0,b]}$ is $L$-bi-Lipschitz and $h|_{[b,b+c]}$ is $1$-bi-Lipschitz. 
Fix now $s \in [0,b]$ and $t \in [b,c]$. Then
\[ d(h(s),h(t)) \le d(f(s),f(b)) + d(g(b),g(t)) \le L(b - s) + t - b \le L(t - s).\]

For the lower bound, we claim that $d(h(t),h(s)) \geq d(h(t),h(b))$. Indeed, if this was not the case, then
\begin{align*} 
\dist(f([0,b]),p) &\leq d(h(s),h(t)) + d(h(t),p)\\ 
&< d(h(b),h(t)) + d(h(t),p)\\ 
&= d(f(b),p)\\
&= \dist(f([0,b]),p)
\end{align*}
which is impossible.
Similarly, $d(h(s),h(t)) \ge d(h(s),h(b))$.
Therefore, 
\begin{align*}
d(h(s),h(t)) \geq \tfrac12 d(h(s),h(b)) + \tfrac12d(h(t),h(b)) &\geq (2L)^{-1}|s-b| + \tfrac12|b-t|\\ 
&\geq (2L)^{-1}|s-t|.\qedhere
\end{align*}
\end{proof}

We are now ready to show Lemma \ref{Prop:Modi}.

\begin{proof}[Proof of Lemma \ref{Prop:Modi}]
Without loss of generality, assume that $\diam{\sigma([0,1])} =1$. Since $X$ is doubling, it is $(C',\alpha)$-homogeneous for some $C'>0$ and $\alpha>0$.

Fix $\e > 0$. If $d(\sigma(0),\sigma(1)) < 2\e$, then we can simply define $\gamma$ to be the geodesic from $\sigma(0)$ to $\sigma(1)$ which is $1$-bi-Lipschitz. Assume now that $d(\sigma(0),\sigma(1)) \geq 2\e$. 

Let $Y \subset \sigma([0,1])$ be a maximal $(\e/4)$-separated set that contains $\sigma(0)$ and $\sigma(1)$. Since $\sigma([0,1])$ is connected, there exists a finite sequence of distinct points $x_0, \dots, x_n$ in $Y$ such that $x_0 = \sigma(0)$, $x_n = \sigma(1)$, and $d(x_{i-1},x_i) < \e/2$ for all $i\in\{1,\dots,n\}$. By the homogeneity of $X$, we have that $n \leq C'(\e/4)^{-\alpha}$.
	
We define a curve $\gamma$ inductively. Let $\gamma_1:[0,s_1] \to X$ be a geodesic with $\gamma_1(0) = \sigma(0)$ and $\gamma_1(s_1) = x_1$. Clearly, $\gamma_1$ is $1$-bi-Lipschitz and for all $t\in [0,s_1]$
\[ \dist(\gamma_1(t), \sigma([0,1]) \leq \text{length}(\gamma_1) \leq \e/2.\]
Suppose that for some $k\in \{1,\dots,n-1\}$ we have defined $s_k>0$ and a $2^{k-1}$-bi-Lipschitz curve $\gamma_{k}:[0,s_{k}] \to X$ parameterized by arc-length, such that $\gamma_{k}(0) = \sigma(0)$, $\gamma_{k}(s_{k}) = x_k$, and $\gamma_{k}([0,s_{k}]) \subset \overline{B}(\sigma([0,1]),\e/2)$. Let $r_{k} \in [0,s_{k}]$ be such that 
\[ c_k := \dist(\gamma_{k}([0,s_{k}]), x_{k+1}) = d(\gamma_{k}(r_k),x_{k+1}).\] 
Define $s_{k+1}=r_k+c_k$ and let $\gamma_{k+1}:[0,s_{k+1}] \to X$ be the concatenation of $\gamma_{k}|_{[0,r_{k}]}$ with a geodesic joining $\gamma_{k}(r_k)$ to $x_{k+1} \in \sigma([0,1])$. Note that 
\[ d(\gamma_{k}(r_k),x_{k+1}) \leq d(x_{k},x_{k+1}) < \e/2,\]
so for each $t \in [0,s_{k+1}]$, we have 
\[ \dist(\gamma_{k+1}(t),\sigma([0,1])) < \e/2.\] 
Moreover, by Lemma \ref{Lem:BL}, the curve $\gamma_{k+1}$ is $2^k$-bi-Lipschitz.

By induction, we have defined a number $0< s_n \leq n(\e/2)$ and a $2^{n-1}$-bi-Lipschitz curve $\gamma_n:[0,s_n] \to X$ such that $\gamma_n(0)=\sigma(0)$, $\gamma_n(s_n)=\sigma(1)$, and
\[ \gamma_n([0,s_n]) \subset B(\sigma([0,1]),\e).\]
The desired curve $\gamma:[0,1]\to X$ is the reparameterization $\gamma(t) = \gamma_n(s_nt)$.
\end{proof}

\subsection{Proof of Proposition \ref{cor:Modi}}

The proof of Proposition \ref{cor:Modi} will rely on the quantitative version of Theorem~\ref{thm:1}; see Theorem~\ref{thm:main2}.

We first review some elementary notions from graph theory. A \emph{(combinatorial) graph} is a pair $G=(V,E)$ of a finite vertex set $V$ and an edge set $E$ which contains elements of the form $\{v,v'\}$ where $v,v' \in V$ and $v\neq v'$. A graph $G' = (V',E')$ is a \emph{subgraph} of $G=(V,E)$ if $V'\subset V$, $E'\subset E$, and $E' \subset V' \times V'$. A \emph{simple path} joining $x,y\in V$ in $G$ is a set $ \gamma = \{v_0, \dots, ,v_n\} \subset V$ of distinct points such that $v_0=x$, $v_n=y$, and $\{v_{i-1},v_{i}\} \in E$ for all $i\in\{1,\dots,n\}$. A graph $G$ is \emph{connected} if any two distinct vertices can be joined by a simple path in $G$.

\begin{lem}\label{lem:2-to-1}
Given a graph $G=(V,E)$ and two distinct $v,v' \in V$, there exists a finite sequence $(v_{i})_{i=1}^N$ in $V$ such that $\{v_1,\dots,v_n\} = V$, $v_1=v$, $v_n = v'$, for each $i\in \{1,\dots,n-1\}$, we have $\{v_i,v_{i+1}\} \in E$, and for each $e\in E$ there exists at most two $i\in \{1,\dots,n-1\}$ such that $e = \{v_i,v_{i+1}\}$.
\end{lem}

\begin{proof}
We will use the fact that every connected graph admits a 2-to-1 Euler tour along its edges, that is, for each vertex $z$ there exists a finite sequence $(z_j)_{j=1}^m$ of vertices in $G$ such that $z_1 = z_m = z$, $\{z_j,z_{j+1}\}$ is an edge for all $j$, and for each edge $e$ there exists exactly two $j$ such that $e = \{z_j,z_{j+1}\}$. See for example the Euler tour technique introduced in \cite{TV84}.

Now let $G,v,v'$ be as in the statement. Deleting some edges from $E$, we may assume that $G$ is a (combinatorial) tree, that is, for any two distinct vertices there exists unique simple path in $G$ that connects them. Let $\tilde{V} = \{v_1,\dots,v_k\}$ be the unique such path with $v_1 =v$ and $v_k=v'$. For each $i\in\{1,\dots,k\}$, let $G_i = (V_i,E_i)$ be the maximal subgraph of $G$ with the property that any simple path connecting a vertex of $G_i$ with a vertex of $\tilde{V}$ must contain $v_i$. Since $G$ is connected, it follows that each $G_i$ is connected. Moreover, since $G$ is a tree, for any $i\neq j$ the graphs $G_i$ and $G_j$ are trees with mutually disjoint vertices (and hence edges).

The construction of the finite sequence $(v_i)_i$ is as follows. Firstly, do a 2-to-1 tour of $G_1$ starting and ending on $v_1$. Then proceed to $v_2$ and do a 2-to-1 tour of $G_2$ starting and ending on $v_2$. Continue in this way until reaching $v_k$ where we do a 2-to-1 tour of $G_k$ starting and ending on $v_k$. 
\end{proof}

\begin{proof}[{Proof of Proposition \ref{cor:Modi}}]
If $\diam{K}=0$, then there is nothing to prove. Assume now that $\diam{K}>0$ and, rescaling, we may further assume that $\diam{K} = 1$.

Let $Y$ be a maximal $(\e/4)$-separated subset of $K$ that contains $x$ and $y$. By the regularity of $X$, the cardinality of $Y$ is at most $C'\e^{-Q}$ for some $C'>0$ depending only on the constants of $Q$-regularity. Define a graph $G$ with vertex set $Y$ such that two points $z,z'\in G$ are connected by an edge if and only if $d(z,z')< \e/2$. Since $K$ is connected, it follows that $G$ is connected. By Lemma \ref{lem:2-to-1} there exists a tour $x=v_0, \dots, v_n=y$ of the vertices $Y$ such that each edge is visited at most twice. 

For each $z \in Y$, denote by $m_z$ the number of indices $i$ such that $v_i = z$. There exists $C''>0$ depending only on the constants of $Q$-regularity such that each vertex of $G$ is contained in at most $C''$ edges. Therefore, for each $z\in Y$, $m_z \leq C''$ and it follows that $n\leq C''C'\e^{-Q}$. Moreover, there exists $c>4$ depending only on the constants of $Q$-regularity such that for each $z\in Y$ there exist points $v_{z,1},\dots,v_{z,m_{z}} \in B(z, \e/16)$ such that
\[ d(v_{z,i},v_{z,j}) \geq c^{-1}\e, \quad\text{for all $z\in Y$ and $i\neq j$}.\]
We may also assume that $v_{x,1}=x$ and $v_{y,m_y} = y$.

Given $i\in \{0,\dots,n\}$ let $j(i)$ be the number of indices $l\in\{0,\dots,i\}$ such that $v_l = v_i$. Define now $\tilde{v}_i =  v_{v_i,j(i)}$. Note that the new sequence $\tilde{v}_0, \dots, \tilde{v}_n$ satisfies 
\begin{enumerate}
\item $\tilde{v}_0 = x$, $\tilde{v}_n =y$, 
\item for each distinct $i,j \in \{0,\dots,n\}$ we have $d(\tilde{v}_i,\tilde{v}_j) \geq c^{-1}\e$,
\item for each $z\in K$ there exists $i\in \{0,\dots,n\}$ such that $d(z,\tilde{v}_i) \leq \e/2$,
\item for each $i \in \{0,\dots,n\}$, $\dist(\tilde{v}_i,K) \leq \e/16$.
\end{enumerate}

Define a map $f: \{i\e : i=0,\dots,n\} \to X$ by $f(i\e)=\tilde{v}_i$ and note that $f$ is $L'$-bi-Lipschitz with $L'=\max\{nc,2/\e\}$. Indeed, for any distinct $i,j$ we have $\e \leq |i\e-j\e| \leq n\e$ and $c^{-1}\e \leq d(f(i\e),f(j\e)) \leq 1+2c^{-1}\e < 2$. 

By Theorem \ref{thm:main2}, there exists a constant $L$ depending on $\e$, the constants of $Q$-regularity, and the data of the Poincar\'e inequality, and there exists a bi-Lipschitz arc $F:[0,n\e] \to X$ that extends $f$ and
\[ \dist_H(K,F([0,n\e])) \lesssim \e. \] 
The arc $\gamma:[0,1] \to X$ in question is obtained by reparameterizing $F$.
\end{proof}

\section{Whitney intervals and a preliminary extension}\label{sec:prelimembeddings}

Here and for the rest of this section we assume that $X$ is a complete geodesic $(C_1,Q)$-Ahlfors regular metric measure space supporting a $p$-Poincar\'e inequality with data $C$ where $p \in (1,Q - 1)$ and $C_1,C>1$. We also assume that $A \subset \R$ is a closed set and that $f: A \to X$ is an $L$-bi-Lipschitz embedding.

Let $I$ be the smallest closed interval with $A\subset I$ (possibly $\R$). We need a Whitney decomposition of $I\setminus A$ as in Whitney's classical proof of his extension theorem \cite{W34}. We may assume that $A$ is not a closed interval itself, as then there is no extension to be made.

\begin{Lem}[{\cite[Theorem VI.1.1,Proposition VI.1.1]{Stein:1970}}]
\label{Thm:Decomp}
There exists a collection of closed intervals $\{\mathcal{Q}_i\}_{i \in \N}$ such that
\begin{enumerate}[(i)]
\item $\bigcup_{i = 1}^{\infty} \mathcal{Q}_i = I \setminus A$,
\item the intervals $\{\mathcal{Q}_i\}$ have disjoint interiors, and
\item $\diam{\mathcal{Q}_i} \le \dist(\mathcal{Q}_i,A) \le 4\diam{\mathcal{Q}_i}$ for all $i \in \mathbb{N}$.
\end{enumerate}
Moreover, if the intervals $\mathcal{Q}_i$ and $\mathcal{Q}_j$ share an endpoint, then
\begin{equation}\label{eq:neighbor}
\tfrac{1}{4}\diam{\mathcal{Q}_j} \le \diam{\mathcal{Q}_i} \le 4\diam{\mathcal{Q}_j}.
\end{equation}
\end{Lem}

Henceforth, the intervals $\{\mathcal{Q}_i\}_{i \in \N}$ will be called \emph{Whitney intervals}.

\subsection{Reference points}\label{sec:endpoints} 

Let $E$ denote the collection of endpoints of $\{\mathcal{Q}_i\}_{i\in\N}$. For each $x \in E$, fix a point $a_x\in A$ that is a closest point of $A$ to $x$, that is, $|x - a_x| = \dist(x,A)$.

\begin{Prop}
\label{Prop:Endpoint}
There exists $\xi \in (0,1)$ and $\tilde L>1$ depending only on $L$, $C_1$, and $Q$, and there exists an $\tilde L$-bi-Lipschitz map $\pi: E \to X$ such that, for all distinct $x,y \in E$, 
\begin{enumerate}
\item $\tfrac14 |x-a_x| \leq d(\pi(x),f(a_x)) \leq 4|x-a_x|$
\item $\dist(\pi(x),f(A)) \geq \xi |x-a_x|$,
\item $d(\pi(x),\pi(y)) \geq \xi\left( |x-a_x| + |y-a_y| \right)$.
\end{enumerate}
Moreover, if $\mathcal{Q}_i = [x,y]$, then 
\begin{align}\label{eq:endpoint} 
d(\pi(x),\pi(y)) \leq d(f(a_x),f(a_y)) + 36\diam{\mathcal{Q}_i} \leq 46L\diam{\mathcal{Q}_i}.
\end{align}
\end{Prop}

We start with a result that allows us to partition $E$ into a finite number of subsets such that elements of the same subset are far apart quantitatively. Recall that by Lemma \ref{lem:por}, there exists $p_0\ge 1$ depending only on $L$, $C_1$, and $Q$ such that $f(A)$ is $p_0$-porous.

\begin{Lem}\label{Lem:Filter}
There exists $n\in\N$ depending only on $L$, $C_1$, and $Q$, and there exists a partition of $E$ into mutually disjoint sets $E_1,\dots,E_n$ such that for any $i \in \{1,\dots,n\}$ and for any $x,y \in E_i$, either
\begin{enumerate}[(F1)]
\item $|x - y| > (12L)\max\{|x - a_x|,|y - a_y|\}$, or
\item $\max\{|x - a_x|,|y - a_y|\} > (8p_0)\min\{|x - a_x|,|y - a_y|\}$.
\end{enumerate}
\end{Lem}
\begin{proof}
Enumerate $E = \{x_1,x_2,\dots\}$, and for each $i \in \N$, define $V_i$ be the set of all indices $j \in \N$ such that
\[ |x_i - x_j| \le (12L)\max\{|x_i - a_{x_i}|,|x_j - a_{x_j}|\} \]
and
\[ (8p_0)^{-1}|x_i - a_{x_i}| \le |x_j - a_{x_j}| \le (8p_0)|x_i - a_{x_i}|. \]
Note that $i\in V_j$ if and only if $j\in V_i$.

We claim that there exists $n \in \N$ depending only on $L$, $C_1$, and $Q$ such that $\card(V_i) \le n$ for each $i \in \N$. To this end, fix $i \in \N$ and note that for any $j, k \in V_i$ with $j\neq k$,
\begin{equation}
\label{eq:X_idiam}
|x_j - x_k| \le (24L)\max\{|x_j - a_{x_j}|,|x_k - a_{x_k}|,|x_i - a_{x_i}|\} \le (192p_0L)|x_i - a_{x_i}|.  
\end{equation}

Moreover, let $j, k \in V_i$ with $j\neq k$, and let $\mathcal{Q}_{j_1}$ and $\mathcal{Q}_{j_2}$ be the two Whitney intervals which share the endpoint $x_j$ and $\mathcal{Q}_{k_1}$ and $\mathcal{Q}_{k_2}$ be the two Whitney intervals which share the endpoint $x_k$.
We have
\[ \dist(\mathcal{Q}_{j_1},A) \ge |x_j - a_{x_j}| - \diam{\mathcal{Q}_{j_1}} \ge |x_j - a_{x_j}| - \dist(\mathcal{Q}_{j_1},A), \]
so $|x_j - a_{x_j}| \le 2\dist(\mathcal{Q}_{j_1},A)$. By Lemma \ref{Thm:Decomp}(iii),
\[ \diam{\mathcal{Q}_{j_1}} \ge \tfrac14\dist(\mathcal{Q}_{j_1},A) \ge \tfrac18|x_j - a_{x_j}|, \]
and a similar estimate holds for $\mathcal{Q}_{j_2}$, $\mathcal{Q}_{k_1}$, and $\mathcal{Q}_{k_2}$. Since $x_j \ne x_k$, one of the intervals for which they are endpoints lies between them. That is,
\begin{align*}
|x_j - x_k| & \ge \min\{\diam{\mathcal{Q}_{j_1}},\diam{\mathcal{Q}_{j_2}},\diam{\mathcal{Q}_{k_1}},\diam{\mathcal{Q}_{k_2}}\} \\
& \ge \tfrac18\min\{|x_j - a_{x_j}|,|x_k - a_{x_k}|\} \\
& \ge (64p_0)^{-1}|x_i - a_{x_i}|.
\end{align*}
Combining this with \eqref{eq:X_idiam},
we conclude that
$\card(V_i) \leq 192L(8p_0)^2 =:n$.

Define now a map $\textbf{c}:\N \to \{1,\dots,n\}$ such that $\textbf{c}(1) = 1$, and for each $i \ge 2$, 
\[ \textbf{c}(i) := \min\{\ell \in \N  : \ell \neq \textbf{c}(k) \text{ for all } k \in V_i \cap \{ 1,\dots,i-1\}\}. \]
It is clear that if $i \in V_j$ and $i \ne j$ then $\textbf{c}(i) \ne \textbf{c}(j)$.
For each $i \in \{1,\dots,n\}$, define $E_i := \{x_j : \textbf{c}(j)=i\}$.
Given $x_j,x_k \in E_i$, $\textbf{c}(i) = \textbf{c}(j)$, so $j \notin V_k$ (equivalently, $k \notin V_j$). Properties (F1) and (F2) follow.
\end{proof}

We now turn to the proof of Proposition \ref{Prop:Endpoint}.

\begin{proof}[{Proof of Proposition \ref{Prop:Endpoint}}]
Let $n\in\N$ and $E_1,\dots,E_n$ be the integer and partition, respectively, from Lemma \ref{Lem:Filter}. For each $k\in \{1,\dots,n\}$, define $E^{(k)} = E_1\cup\cdots\cup E_k$.

Let $i \in \{1,\dots,n\}$, $x \in E_i$, and $x' \in \partial B(f(a_x),|x - a_x|)$. By the porosity of $f(A)$ there exists a point $\tilde{x} \in X$ such that 
\[ B(\tilde{x},(2p_0)^{-1}|x - a_x|) \subset B\left(x',\tfrac12|x - a_x|\right) \setminus f(A).\] 
Then
\begin{equation}
\label{eq:sphere}
\tfrac12|x - a_x| \le d(\tilde{x},f(a_x)) \le \tfrac32|x - a_x|
\end{equation}
and
\begin{equation}
\label{eq:porous}
\dist(\tilde{x},f(A)) \ge (2p_0)^{-1}|x - a_x|.
\end{equation}

For any $i \in \{1,\dots,n\}$, and for any $x,y \in E_i$, we will show that
\begin{equation}\label{eq:eta}
d(\tilde{x},\tilde{y}) \geq (8p_0)^{-1}\left(|x - a_x| + |y - a_y|\right).
\end{equation}
Fix such $i$, $x$, and $y$, and assume without loss of generality that $|x - a_x| \ge |y - a_y|$.
If (F1) holds, then by \eqref{eq:sphere},
\begin{align*}
d(\tilde{x},\tilde{y}) & \ge d(f(a_x),f(a_y)) - d(f(a_x),\tilde{x}) - d(f(a_y),\tilde{y}) \\
& \ge L^{-1}|a_x - a_y| - \tfrac32(|x - a_x| +|y - a_y|) \\
& \ge L^{-1}|x - y| - 6|x - a_x| \\
& > 6|x - a_x| \\
& \ge 3(|x - a_x| + |y - a_y|).
\end{align*}
Assume now that (F2) holds and (F1) fails. By \eqref{eq:sphere} and \eqref{eq:porous},
\begin{align*}
d(\tilde{x},\tilde{y}) & \ge d(\tilde{x},f(a_y)) - d(f(a_y),\tilde{y}) \\
& \ge (2p_0)^{-1}|x - a_x| - \tfrac32|y - a_y| \\
& \ge (8p_0)^{-1}(|x - a_x| + |y - a_y|).
\end{align*}

We define the map $\pi$ on each $E^{(k)}$ in an inductive manner. Define $\pi:E_1 \to X$ by $\pi(x) = \tilde{x}$. Properties (1)--(3) of the proposition for $E_1$ follow from \eqref{eq:sphere}, \eqref{eq:porous}, and \eqref{eq:eta} with $\xi_1=(8p_0)^{-1}$.

Assume now that for some $k\in\{1,\dots,n-1\}$ we have defined a constant $\xi_{k} \in (0,1)$ and a function $\pi: E^{(k)} \to X$ such that for all distinct $x,y \in E^{(k)}$,
\begin{equation}
\label{eq:sphere2}
\tfrac14|x - a_x| \le d(\pi(x),f(a_x)) \le 4|x - a_x|,
\end{equation}
\begin{equation}
\label{eq:porous2}
\dist(\pi(x),f(A)) \ge (4p_0)^{-1}|x - a_x|,
\end{equation}
and
\begin{equation}
\label{eq:apart2}
d(\pi(x),\pi(y)) \ge \xi_{k}(|x - a_x| + |y - a_y|).
\end{equation}

Fix $x \in E_{k + 1}$, and assume that there exist $y_1,\dots,y_{N} \in E^{(k)}$ such that
\begin{equation}
\label{eq:close}
d(\tilde{x},\pi(y_j)) < (8p_0)^{-1}(|x - a_x| + |y_j - a_{y_j}|)
\end{equation}
for each $j \in \{1,\dots,N\}$. First, for each such $j$, by \eqref{eq:sphere2}, \eqref{eq:porous}, and \eqref{eq:close},
\begin{align*}
|y_j - a_{y_j}| & \ge \tfrac14d(\pi(y_j),f(a_{y_j})) \\
& \ge \tfrac14 \left(d(\tilde{x},f(a_{y_j})) - d(\tilde{x},\pi(y_j))\right) \\
& \ge \tfrac14 \left( (2p_0)^{-1}|x - a_x| - (8p_0)^{-1}(|x - a_x| + |y_j - a_{y_j}|)\right).
\end{align*}
This gives
\begin{equation}
\label{eq:great}
|y_j - a_{y_j}| \ge (12p_0)^{-1}|x - a_x|.
\end{equation}
Next, for any $j,\ell \in \{1,\dots,N\}$, \eqref{eq:apart2} and \eqref{eq:great} yield
\begin{equation}
\label{eq:sep}
d(\pi(y_j),\pi(y_\ell)) \ge (12p_0)^{-1}\xi_{k}|x - a_x|,
\end{equation}
so $\{\pi(y_1),\dots,\pi(y_{N})\}$ is a $((12p_0)^{-1}\xi_{k}|x - a_x|)$-separated set. By \eqref{eq:sphere}, \eqref{eq:porous2}, and \eqref{eq:close}, for any $j \in \{1,\dots,N\}$,
\begin{align*}
|x - a_x| & \ge \tfrac23 d(\tilde{x},f(a_x)) \\
& \ge \tfrac23 \left( d(\pi(y_j),f(a_x)) - d(\tilde{x},\pi(y_j))\right) \\
& \ge (6p_0)^{-1}|y_j - a_{y_j}| - (12p_0)^{-1}(|x - a_x| + |y_j - a_{y_j}|).
\end{align*}
Therefore,
\begin{equation}
\label{eq:less}
|y_j - a_{y_j}| \le 24p_0|x - a_x|.
\end{equation}
By Ahlfors regularity, \eqref{eq:close}, \eqref{eq:less}, and \eqref{eq:sep}, we obtain $N \lesssim_{L,C_1,Q} \xi_k^{-Q}$.

By Lemma \ref{lem:por}, $\{\pi(y_1),\dots,\pi(y_{N})\}$ is $p_k$-porous for some $p_k \ge 1$ depending only on $C_1$, $Q$, $L$, and $k$. Hence, there exists a point $\pi(x) \in B(\tilde{x},(32p_0)^{-1}|x - a_x|)$ such that
\begin{equation}
\label{eq:new}
B(\pi(x),(32p_0p_k)^{-1}|x - a_x|) \subset X \setminus \{\pi(y_1),\dots,\pi(y_{N})\}.
\end{equation}

To complete the inductive step, we show that $\pi$ defined on $E^{(k+1)}$ satisfies properties (1)--(3) of the proposition for some appropriate $\xi_{k+1} \in (0,1)$.

For the first property, fix $x\in E_{k+1}$. By \eqref{eq:sphere}, 
\[ d(\pi(x),f(a_x)) \leq d(\pi(x),\tilde{x}) + d(f(a_x),\tilde{x})\leq 2|x-a_x|\]
and
\[ d(\pi(x),f(a_x)) \geq  d(f(a_x),\tilde{x}) - d(\pi(x),\tilde{x})  \geq \tfrac14|x-a_x|. \]

For the second property, fix $x\in E_{k+1}$. By \eqref{eq:porous}
\[ \dist(\pi(x),f(A)) \geq \dist(\tilde{x},f(A)) - d(\pi(x),\tilde{x}) \geq (4p_0)^{-1}|x-a_x|.\]

For the third property, fix distinct $x,y \in E^{(k+1)}$ and assume that $x\in E_{k+1}$. If $y \in E_{k+1}$, then by \eqref{eq:eta}
\begin{align*} 
d(\pi(x),\pi(y)) &\geq d(\tilde{x},\tilde{y}) - d(\pi(x),\tilde{x}) - d(\pi(y),\tilde{y})\geq (16p_0)^{-1}(|x-a_x| + |y-a_y|).
\end{align*}
Assume now that $y\in E^{(k)}$. If
\[ d(\tilde{x},\tilde{y}) \geq (8p_0)^{-1}\left( |x-a_x| + |y-a_y| \right),\]
then we work as in the preceding case. If 
\[ d(\tilde{x},\tilde{y}) < (8p_0)^{-1}\left( |x-a_x| + |y-a_y| \right),\]
then by \eqref{eq:less} and \eqref{eq:new},
\begin{align*}
d(\pi(x),\pi(y)) &\geq (32p_0p_k)^{-1}|x-a_x| \geq (2^93p_0^2p_k)^{-1}\left( |x-a_x| + |y-a_y| \right).
\end{align*}

After $n$ steps, we have defined $\xi:=\xi_n$ and the map $\pi:E \to X$ that satisfies properties (1)--(3) in the statement of the lemma.

To show that $\pi$ is bi-Lipschitz, fix distinct $x,y \in E$ and assume without loss of generality that $|x-a_x|\geq |y-a_y|$. By Lemma \ref{Thm:Decomp}(iii), $|x-a_x| \leq 4|x-y|$. 
By property (3),
\begin{align*}
d(\pi(x),\pi(y)) &\leq d(\pi(x),f(a_x)) + d(f(a_x),f(a_y)) + d(\pi(y),f(a_y))\\
&\leq 4|x-a_x| + L|a_x-a_y| + 4|y-a_y|\\
&\leq 41L|x-y|.
\end{align*}
For the lower bound, suppose first that $|a_x-a_y|\geq 16L|x-a_x|$. Then, $|x-y| \leq 2|a_x-a_y|$
and by property (1)
\[ d(\pi(x),\pi(y)) \geq d(f(a_x),f(a_y)) - d(\pi(x),f(a_x)) - d(\pi(y),f(a_y)) \geq (2L)^{-1}|a_x-a_y|.\]
Suppose now that $|a_x-a_y| \leq 16L|x-a_x|$. Then, $|x-y| \leq (2+16L)|x-a_x|$ and by property (3), $d(\pi(x),\pi(y)) \geq \xi |x-a_x|$.

For \eqref{eq:endpoint}, fix a Whitney interval $\mathcal{Q}_i = [x,y]$ and assume, without loss of generality, that $|x-a_x|\leq |y-a_y|$. There are two cases to consider. Assume first that $a_x=a_y$. By property (1),
\begin{align*}
d(\pi(x),\pi(y)) \leq 4|x-a_x| + 4|y-a_x| \leq 8|x-a_x| + 4\diam{\mathcal{Q}_i} \leq 36\diam{\mathcal{Q}_i}.
\end{align*}
Assume now that $a_x \neq a_y$. Then
\[ |y-a_y| \leq |y-a_x| \leq |x-y|+ |x-a_x| \leq 5\diam{\mathcal{Q}_i}\]
which yields that $
|a_x-a_y| \leq 11\diam{\mathcal{Q}_i}$. By property (1),
\begin{align*}
d(\pi(x),\pi(y)) &\leq 4|x-a_x| + d(f(a_x),f(a_y)) + 4|y-a_y|\\
&\leq d(f(a_x),f(a_y)) + 40\diam{\mathcal{Q}_i}\\
&\leq L|a_x-a_y| + 40\diam{\mathcal{Q}_i}\\
&\leq 51L\diam{\mathcal{Q}_i}.\qedhere
\end{align*}
\end{proof}

\subsection{The middle third of each Whitney interval}\label{sec:middlethird} 

The goal of this subsection is to extend $f$ to the union of the middle-thirds of all Whitney intervals $\{\mathcal{Q}_i\}_{i \in \N}$ in a bi-Lipschitz way. From here on, for each Whitney interval $\mathcal{Q}_i$, we denote by $\hat{\mathcal{Q}}_i$ the middle third interval of $\mathcal{Q}_i$. Recall the constants $\xi\in (0,1)$ and $\tilde L$ from Proposition \ref{Prop:Endpoint} depending only on $L$, $C_1$, and $Q$.

\begin{Prop}\label{Prop:middlethird}
There exists a constant $\hat{L} \geq 1$ depending only on $p$, $C$, $C_1$, $L$, and $Q$, and there exists an $\hat{L}$-bi-Lipschitz extension of $f$ 
\[ g: A \cup \bigcup_{i\in\N}\hat{\mathcal{Q}}_i \to X\]
such that for each $i\in\N$, if $\mathcal{Q}_i = [w,z]$, and $\hat{\mathcal{Q}}_i = [\hat{w},\hat{z}]$, then
\begin{enumerate}
\item $d(g(\hat{w}),\pi(w)) \leq (2^8\tilde L)^{-1}\xi\diam{\mathcal{Q}_i}$,
\item $d(g(\hat{z}),\pi(z)) \leq 
 (2^8\tilde L)^{-1}\xi\diam{\mathcal{Q}_i}$, and
\item $g(\hat{\mathcal{Q}}_i) \subset B\left(\pi(w),4R_i\right) \cap B\left(\pi(z),4R_i\right)$, where $R_i = d(\pi(w),\pi(z))$.
\end{enumerate}
\end{Prop}

Recall that, since $f$ is bi-Lipschitz, the set $f(A)$ is $1$-homogeneous in $X$.

\begin{Lem}\label{Lem:modulus}
There exist constants $\beta_0, \ell_0, \delta_0 > 0$ depending only on $p$, $C$, $C_1$, $L$, and $Q$ with the following property. Let $\mathcal{Q}_i = [w,z]$ be a Whitney interval and $\Gamma_i$ be the collection of curves $\gamma:[0,1] \to X$ such that
\begin{enumerate}
\item $\gamma([0,1]) \subset B(\pi(w),3R_i)\cap B(\pi(z),3R_i)$ where $R_i = d(\pi(w),\pi(z))$,
\item $\max\{ d(\gamma(0),\pi(w)), d(\gamma(1),\pi(z))\}<(2^8\tilde L)^{-1}\xi \diam{\mathcal{Q}_i}$,
\item $\len(\gamma) \le \ell_0 \diam{\mathcal{Q}_i}$,
\item $\dist(\gamma(t),f(A)) \ge \delta_0 \diam{\mathcal{Q}_i}$ for all $t\in[0,1]$.
\end{enumerate}
Then, 
\[\Mod_p(\Gamma_i) \geq \beta_0 (\diam{\mathcal{Q}_i})^{Q-p}.\]
\end{Lem}

\begin{proof}
Since $B(\pi(w),2R_i) \subset B(\pi(w),3R_i)\cap B(\pi(z),3R_i)$,
we may apply Lemma \ref{Lem:Path1}, Proposition \ref{Prop:Endpoint}(3), and \eqref{eq:endpoint} to conclude that the family $\Gamma_i^{(1)}$ of curves 
\[ \gamma:[0,1] \to B(\pi(w),3R_i)\cap B(\pi(z),3R_i)\] 
such that $\gamma(0)$ lies in the closed ball $\overline{B}(\pi(w),(2^8\tilde L)^{-1}\xi \diam{\mathcal{Q}_i})$ and $\gamma(1)$ lies in the closed ball $\overline{B}(\pi(z),(2^8\tilde L)^{-1}\xi \diam{\mathcal{Q}_i})$ has $p$-modulus
\[ \Mod_p(\Gamma_i^{(1)}) \ge \alpha(\diam{\mathcal{Q}_i})^{Q - p} \]
where $\alpha > 0$ is some constant depending only on $p$, $C$, $C_1$, $Q$, and $L$. 

By Lemma \ref{Lem:Path2}, there exists $\ell_0>0$ depending only on $p$, $C$, $C_1$, $Q$, and $L$ such that the subfamily
\[ \Gamma_i^{(2)} := \left\{\gamma \in \Gamma_i^{(1)} : \len(\gamma) \le \ell_0 \diam{\mathcal{Q}_i}\right\} \]
satisfies
\begin{align*}
\Mod_p(\Gamma_i^{(2)}) \ge 
\tfrac12 \alpha (\diam{\mathcal{Q}_i})^{Q - p}.
\end{align*}
By Lemma \ref{Lem:Path3}, there exists $\delta_0>0$ depending only on $Q$, $p$, $C$, $C_1$, and $L$ such that
the subfamily
\[ \Gamma_i := \left\{\gamma \in \Gamma_i^{(2)} : \dist(\gamma(t),f(A)) \ge \delta_0 \diam{\mathcal{Q}_i} \text{ for each } t \in [0,1]\right\}, \]
satisfies 
\[ \Mod_p(\Gamma_i) \ge \tfrac14 \alpha (\diam{\mathcal{Q}_i})^{Q - p}. \qedhere\] 
\end{proof}

We now need a filtration of the Whitney decomposition, in the vein of the following result of David and Semmes. The proof of the lemma is almost identical to that of Lemma \ref{Lem:Filter} and is left to the reader. 

\begin{Lem}[{\cite[Proposition 17.4]{David:1991}}]
\label{Lem:DS}
There exists an integer $N$ depending only on $L$, $C_1$, and $Q$, and there exists a partition of $\N$ into sets $\{\mathscr{I}_1,\dots,\mathscr{I}_N\}$ such that for any $k \in \{1,\dots,N\}$ and for any $i,j \in \mathscr{I}_k$, either
\begin{enumerate}[(i)]
\item $\dist(\mathcal{Q}_i,\mathcal{Q}_j) > 800L^2\max\{\diam{\mathcal{Q}_i},\diam{\mathcal{Q}_j}\}$, or
\item $\max\{\diam{\mathcal{Q}_i} , \diam{\mathcal{Q}_j}\} > 800L\delta_0^{-1}\min\{\diam{\mathcal{Q}_i} , \diam{\mathcal{Q}_j}\}$.
\end{enumerate}
\end{Lem}

We are now ready to prove Proposition \ref{Prop:middlethird}.

\begin{proof}[{Proof of Proposition \ref{Prop:middlethird}}]
The construction is in an inductive fashion. Let $N$ and $\mathscr{I}_1,\dots,\mathscr{I}_N$ be the integer and sets of indices from Lemma \ref{Lem:DS}. Denote $A_0 := A$ and for each $k\in \{1,\dots,N\}$ denote
\[ A_k := A_0 \cup \bigcup_{j = 1}^{k} \bigcup_{i \in \mathscr{I}_j} \hat{\mathcal{Q}}_i. \]

For each $k\in \{0,\dots,N\}$, we find some $L_k \geq 1$ depending only on $p$, $C$, $C_1$, $L$, $Q$, and $k$, and we find an $L_k$-bi-Lipschitz embedding $f_k :A_k \to X$ such that for all $k\in\{1,\dots,N\}$, $f_k|_{A_{k-1}} = f_{k-1}$ and such that, if $i\in \mathscr{I}_k$, $\mathcal{Q}_i = [w,z]$, and $\hat{\mathcal{Q}}_i = [\hat{w},\hat{z}]$, then
\begin{enumerate}[(a)]
\item $d(f_k(\hat{w}),\pi(w)) \leq (2^8\tilde L)^{-1}\xi\diam{\mathcal{Q}_i}$, 
\item $d(f_k(\hat{z}),\pi(z)) \leq (2^8\tilde L)^{-1}\xi\diam{\mathcal{Q}_i}$,
\item $f_{k}(\hat{\mathcal{Q}}_i) \subset B\left(\pi(w),4R_i\right) \cap B\left(\pi(z),4R_i\right)$ where $R_i = d(\pi(w),\pi(z))$.
\end{enumerate}
The map $g$ of Proposition \ref{Prop:middlethird} will then be the map $f_N$.

For $k=0$, set $L_0 = L$ and $f_0 = f$. Properties (a)--(c) are vacuous.  

Assume now that for some $k \in \{0,\dots,N-1\}$, there exists a constant $L_k$ and an $L_k$-bi-Lipschitz map $f_k:A_k \to X$ satisfying (a)--(c).

Fix $i \in \mathscr{I}_{k + 1}$ and write $\mathcal{Q}_i=[w,z]$ and $\hat{\mathcal{Q}}_i=[\hat{w},\hat{z}]$. Recall the family of curves $\Gamma_i$ from Lemma \ref{Lem:modulus}. By Lemma \ref{Lem:Path3}, there exists $\delta_{k+1}\in (0,\xi)$ depending only on $Q$, $p$, $C$, $C_1$, $L$, and $k$ (in particular on the homogeneity constant of $f(A_k)$) such that the subfamily
\[ \Gamma_{k,i}' := \left\{\gamma \in \Gamma_i: \dist(\gamma(t),f_k(A_k)) \ge \delta_{k+1} \diam{\mathcal{Q}_i} \text{ for each } t \in [0,1]\right\}, \]
satisfies 
\[ \Mod_p(\Gamma_{k,i}') \ge \tfrac12\beta_0 (\diam{\mathcal{Q}_i})^{Q - p} >0.\] 
In particular, $\Gamma_{k,i}'$ is nonempty, so we can pick a curve $\sigma_i \in \Gamma_{k,i}'$. Applying Lemma~\ref{Prop:Modi} to $\sigma_i$ with a suitable reparameterization,
we find a constant $L_{k+1}'$ depending only on $Q$, $p$, $C$, $C_1$, $L$, and $k$, and we find an $L_{k+1}'$-bi-Lipschitz curve $\gamma_i : \hat{\mathcal{Q}}_i \to X$ such that $\gamma_i(\hat{w}) = \sigma_i(0)$, $\gamma_i(\hat{z}) = \sigma_i(1)$, and 
inductive hypothesis (c) for $f_k$ gives
\begin{align}\label{eq:excess}
\gamma_i(\hat{\mathcal{Q}}_i) &\subset B\left(\sigma_i([0,1]),  \tfrac12\delta_{k+1} \diam{\mathcal{Q}_i}\right)\\ 
&\subset B(\pi(w),3R_i + \tfrac12\delta_{k+1}\diam{\mathcal{Q}_i})\cap B(\pi(z),3R_i + \tfrac12\delta_{k+1}\diam{\mathcal{Q}_i})\notag\\
& \subset B(\pi(w),4R_i) \cap B(\pi(z),4R_i)\notag. 
\end{align}
In particular, we have that
$
\dist(\gamma_i(\hat{\cQ}_i),f_k(A_k)) \geq \tfrac12\delta_{k+1} \diam{\mathcal{Q}_i}
$.

Define now $f_{k+1} : A_{k+1} \to X$ by setting $f_{k+1}|A_k = f_k$ and $f_{k+1}|\hat{\mathcal{Q}}_i = \gamma_{i}$ for each $i\in \mathscr{I}_{k+1}$. By \eqref{eq:endpoint} we have for all $i\in\mathscr{I}_{k+1}$
\begin{equation}\label{eq:diameter}
\diam{f_{k+1}(\mathcal{\hat Q}_i)} \leq 
9R_i
\leq 414 L \diam(\cQ_i).
\end{equation}

Clearly, $f_{k+1}|A_k = f_k$. Properties (a)--(c) are clear from the design of $f_{k + 1}$ and Lemma \ref{Lem:modulus}. To complete the inductive step, we claim that $f_{k+1}$ is $L_{k + 1}$-bi-Lipschitz for some $L_{k + 1}\geq 1$ depending only on $Q$, $p$, $C$, $C_1$, $L$, and $k$. Fix $x,y \in A_{k+1}$.

Firstly, if $x,y\in A_k$, then the claim follows by the fact that $f_{k+1}|A_k = f_k$ and the inductive hypothesis that $f_k$ is $L_k$-bi-Lipschitz. 

Secondly, assume that $x \in \hat{\mathcal{Q}}_i$ for some $i\in\mathscr{I}_{k+1}$ and $y \in A$. Let $w$ be the endpoint of $\mathcal{Q}_i$ closest to $A$, let $\hat{w}$ be the endpoint of $\hat{\mathcal{Q}}_i$ between $x$ and $w$, and note that $|w-x| \leq |x-a_w| \leq |x-y|$. By \eqref{eq:diameter}, Proposition \ref{Prop:Endpoint}(1), the fact $\diam{\mathcal{Q}_i} \leq |w-a_w|$, and properties (a), (b) for $f_{k+1}$.
\begin{align*}
d(&f_{k+1}(x),f_{k+1}(y)) \\
&\leq d(f_{k+1}(x),f_{k+1}(\hat{w})) + d(f_{k+1}(\hat{w}),\pi(w)) + d(\pi(w),f(a_w)) + d(f(a_w),f(y))\\
&\leq (414L + 5)|w-a_w| +  L|a_w-y|\\
&\leq (414L + 5)|x-a_w| + L|a_w-y|\\
&\leq (416L + 5)|x-y|.
\end{align*}
For the lower bound, we have by Lemma \ref{Lem:modulus}(4) and the design of $\gamma_i$
\begin{align*}
d(f_{k+1}(x),f_{k+1}(y))
\geq
\dist(f_{k+1}(x),f(A)) \geq \tfrac12\delta_0\diam\cQ_i 
&\geq \tfrac18\delta_0|w-a_w|,
\end{align*}
and, by \eqref{eq:diameter}, property (c) for $f_{k+1}$, and Proposition \ref{Prop:Endpoint}(2)
\begin{align*}
d(f_{k+1}(x),f(a_w)) &\leq d(f_{k+1}(x),f_{k+1}(\hat w)) + d(f_{k+1}(\hat w),\pi(w)) + d(f(a_w),\pi(w))\\
& \le 414L\diam\cQ_i + (2^8\tilde L)^{-1}\xi\diam\cQ_i + 4|w - a_w| \\
& \le 419L|w - a_w|.
\end{align*}
Therefore, since $|x-a_w| \le 2|w - a_w|$,
\begin{align*}
|x - y| &\le |x-a_w| + |a_w-y|\\
&\leq 2|w-a_w| + L[d(f(a_w),f_{k+1}(x)) + d(f_{k+1}(x),f(y))]\\
&\leq 419 L^2 |w-a_w| + Ld(f_{k+1}(x),f(y))\\
&\leq 3352 L^2\delta_0^{-1}d(f_{k+1}(x),f_{k+1}(y))
\end{align*}

Thirdly, assume that $x \in \hat{\mathcal{Q}}_i$ and $y \in \hat{\mathcal{Q}}_j$, for some $i,j \in \mathscr{I}_1 \cup \cdots \cup \mathscr{I}_{k + 1}$. Assume that $\diam{\mathcal{Q}_i} \geq \diam{\mathcal{Q}_j}$. For the upper bound, note that 
\begin{equation}
\label{Eq:T_iT_j}
|x - y| \ge \dist(\hat{\mathcal{Q}}_i,\hat{\mathcal{Q}}_j) \ge \diam{\hat{\mathcal{Q}}_i} + \diam{\hat{\mathcal{Q}}_j} = \tfrac13 \left( \diam{\mathcal{Q}_i} + \diam{\mathcal{Q}_j} \right).
\end{equation}
Let $a_i$ be the closest point of $A$ to $\mathcal{Q}_i$, let $a_j$ be the closest point of $A$ to $\mathcal{Q}_j$, let $e_i$ be the endpoint of $\mathcal{Q}_i$ that lies between $x$ and $a_i$, and let $e_j$ be the endpoint of $\mathcal{Q}_j$ that lies between $y$ and $a_j$. By Proposition \ref{Prop:Endpoint}(1), \eqref{eq:excess}, \eqref{eq:endpoint}, Lemma \ref{Thm:Decomp}(iii), and \eqref{Eq:T_iT_j},
\begin{align*}
d(f_{k+1}(x),f_{k+1}(y)) & \le d(f_{k+1}(x),\pi(e_i)) + d(\pi(e_i),f(a_i)) + d(f(a_i),f(a_j)) \\
& \qquad + d(f(a_j),\pi(e_j)) + d(\pi(e_j),f_{k+1}(y)) \\
&\leq (16 + 184L)(\diam{\mathcal{Q}_i} + \diam{\mathcal{Q}_j}) + L|a_i-a_j|\\
&\leq (16 + 189L)(\diam{\mathcal{Q}_i} + \diam{\mathcal{Q}_j}) + L|x-y|\\
& \le 616L|x - y|
\end{align*}
since $|x-a_i| \le \diam{\mathcal{Q}_i} + |e_i-a_i| \leq 5\diam{\mathcal{Q}_i}$
and, similarly, $|y-a_j| \le 5\diam{\mathcal{Q}_j}$.

For the lower bound, there are two cases to consider.

\emph{Case 1:} $\dist(\mathcal{Q}_i,\mathcal{Q}_j) > 800L^2\diam{\mathcal{Q}_i}$.
By Proposition~\ref{Prop:Endpoint}(1), \eqref{eq:excess}, and Lemma~\ref{Thm:Decomp}(iii),
\begin{align*}
d(f_{k+1}(x),f_{k+1}(y)) & \ge d(f(a_i),f(a_j)) - d(f(a_i),\pi(e_i)) - d(\pi(e_i),f_{k+1}(x))\\
& \qquad \qquad  - d(f(a_j),\pi(e_j)) - d(\pi(e_j),f_{k+1}(y)) \\
& \ge L^{-1}|a_i - a_j| - (184L + 16)(\diam{\mathcal{Q}_i} + \diam{\mathcal{Q}_j}) \\
& \ge L^{-1}|x - y| - L^{-1}(|x - a_i| + |a_j - y|) - 400L\diam{\mathcal{Q}_i} \\
& \ge L^{-1}|x - y| - (10L^{-1} + 400L)\diam{\mathcal{Q}_i} \\
& > L^{-1}|x - y| - 410L (800L^2)^{-1}\dist(\mathcal{Q}_i,\mathcal{Q}_j) \\
& \ge (3L)^{-1}|x - y|.
\end{align*}

\emph{Case 2:} $\dist(\mathcal{Q}_i,\mathcal{Q}_j) \leq 800L^2\diam{\mathcal{Q}_i}$.
In this case, we have
\[ 
|x - y| \le \diam{\mathcal{Q}_i} + \dist(\mathcal{Q}_i,\mathcal{Q}_j) + \diam{\mathcal{Q}_j} \leq 802L^2\diam\cQ_i. 
\]
Case 2 splits now into two subcases.

\emph{Case 2.1:} $i \in \mathscr{I}_{k + 1}$ and $j \in \mathscr{I}_1 \cup \cdots \cup \mathscr{I}_k$.
According to the line following
\eqref{eq:excess},
\begin{align*}
d(f_{k+1}(x),f_{k+1}(y)) \ge d(f_{k+1}(x),f_{k}(A_k))
&\ge \tfrac12\delta_{k+1} \diam{\mathcal{Q}_i} \\
&\ge \delta_{k + 1}(1604L^2)^{-1}|x - y|.
\end{align*}

\emph{Case 2.2:} $i,j \in \mathscr{I}_{k + 1}$. By Lemma~\ref{Lem:DS} we have that 
$\diam{\mathcal{Q}_i} > 800 L\delta_0^{-1}\diam{\mathcal{Q}_j}$. By Lemma \ref{Lem:modulus}(4), the design of $\gamma_i$, Proposition \ref{Prop:Endpoint}(1), and \eqref{eq:excess},
\begin{align*}
d(f_{k+1}(x),f_{k+1}(y)) & \ge \dist(f_{k+1}(\hat{\mathcal{Q}}_i),f_{k+1}(y)) \\
& \ge \dist(f_{k+1}(\hat{\mathcal{Q}}_i),f(a_j)) - d(\pi(e_j),f(a_j))  - d(\pi(e_j),f_{k+1}(y)) \\
& \geq \tfrac12\delta_0 \diam{\mathcal{Q}_i} - (16 + 184L) \diam{\mathcal{Q}_j}\\
& \ge \tfrac14 \delta_0 \diam{\mathcal{Q}_i}\\
& \ge \tfrac14 \delta_0(802L^2)^{-1}|x - y|. \qedhere
\end{align*}
\end{proof}

\section{Proof of Theorem~\ref{thm:1}}\label{sec:proof}

In this section, we will give the proof of the following quantitative version of Theorem~\ref{thm:1}.

\begin{Thm}
\label{thm:main2}
Given $C,C_1>0$, $Q>2$, $p\in (1,Q-1)$, and $L\geq 1$, there exists $L'\geq 1$ with the following property.

Let $(X,d,\mu)$ be a complete geodesic $(C_1,Q)$-Ahlfors regular metric measure space supporting a $p$-Poincar\'e inequality with data $C$. Let $A\subset \R$ be a closed set, let $I$ be the smallest closed interval of $\R$ containing $A$, and let $f:A \to X$ be an $L$-bi-Lipschitz embedding. Then there exists an $L'$-bi-Lipschitz extension $F:I \to X$ of $f$. 

Moreover, if $(x,y)$ is a component of $I\setminus A$, then
\begin{equation} 
\diam{F([x,y])} \leq 75\max\{|x-y|, d(f(x),f(y))\}.
\end{equation}
\end{Thm}

The remainder of this section is devoted to the proof of this theorem. Let $\{\mathcal{Q}_i\}_{i\in\N}$ be the Whitney decomposition of $I \setminus A$ from Lemma \ref{Thm:Decomp} and let
\[\hat{A} := A \cup \bigcup_{i\in\N}\hat{\mathcal{Q}}_i.\]
Recall that $\hat{\mathcal{Q}}_i$ denotes the middle third of the Whitney interval $\mathcal{Q}_i$ and that $E$ denotes the set of endpoints of Whitney intervals $\{\mathcal{Q}_i\}_{i\in\N}$.

There is a map $\pi:E \to X$ satisfying the properties of 
Proposition~\ref{Prop:Endpoint}, 
and there exist a constant $\hat{L}\geq 1$ depending only on $C$, $C_1$, $Q$, $p$, and $L$, and there exists an $\hat{L}$-bi-Lipschitz extension
\[ g : \hat{A} \to X\]
of $f$ 
satisfying the properties outlined in Proposition~\ref{Prop:middlethird}.
In particular, 
if $(x,y)$ is a component of $I\setminus A$, if $\mathcal{Q}_i \subset (x,y)$, and if $x$ is the closest point of $A$ to $\mathcal{Q}_i$, then by \eqref{eq:endpoint} and \eqref{eq:excess},
\begin{equation}\label{eq:thm1}
\max_{z\in \hat{\mathcal{Q}}_i} d(f(x),g(z)) \leq 2d(f(x),f(y)) + 73\diam{\mathcal{Q}_i}.
\end{equation} 

We introduce several pieces of notation.
Given $x\in E$, we denote by $\mathcal{L}_x$ (resp. $\mathcal{R}_x$) the Whitney interval for which $x$ is the right (resp. left) endpoint. As above, $\hat{\mathcal{L}}_x$ and $\hat{\mathcal{R}}_x$ are the middle thirds of intervals $\mathcal{L}_x$ and $\mathcal{R}_x$. By \eqref{eq:neighbor}, for any $x\in E$ we have
\[ \tfrac14\diam{\mathcal{L}_x} \leq \diam{\mathcal{R}_x} \leq 4\diam{\mathcal{L}_x}.\]
Further, for any $x\in E$ we write
\[ \mathcal{L}_x = [x_L,x], \quad \hat{\mathcal{L}}_x = [\tau_x^1,\tau_x^2], \quad \mathcal{R}_x = [x,x_R], \quad \hat{\mathcal{R}}_x = [\tau_x^3,\tau_x^4].\]

\begin{figure}[ht]
    \centering
    \includegraphics[width=0.7\textwidth]{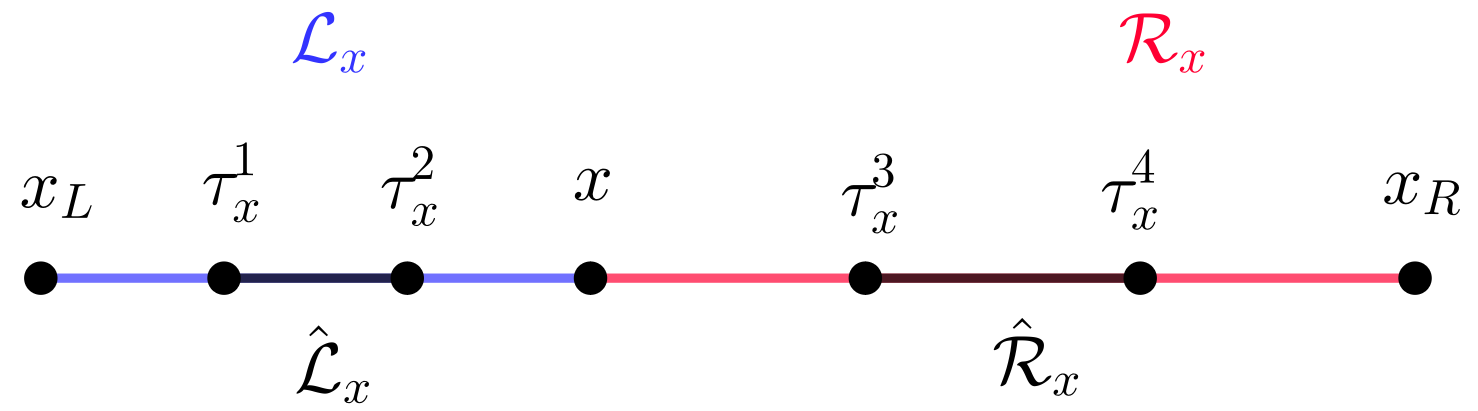}
    \caption{}
    \label{fig:1}
\end{figure}

Since $g$ is $\hat{L}$-bi-Lipschitz, there exists $C_2>0$ depending only on $C$, $C_1$, $Q$, $p$, and $L$ such that the set $g(\hat{A})$ (and each of its subsets) is $(C_2,1)$-homogeneous.

\subsection{Local modifications around points in \texorpdfstring{$E$}{E}}
We divide $E$ into two sets $E'$ and $E''$ such that for any two points in $E'$ there exists a point in $E''$ between them and vice-versa. That is, for any $x\in E'$ we have $x_L,x_R\in E''$, and for any $x\in E''$ we have $x_L,x_R\in E'$.

We perform local modifications around points in $E$ starting with points in $E'$.

\subsubsection{Local modifications around points in $E'$}\label{sec:E1} Fix a point $x\in E'$. By the $(C_2,1)$-homogeneity of $g(\hat{A}\setminus (\hat{\mathcal{L}}_x\cup\hat{\mathcal{R}}_x))$, by Corollary \ref{Cor:Path} and  Proposition~\ref{Prop:middlethird}(1,2), there exists a constant $C'\geq 1$ depending only on $C$, $C_1$, $Q$, $p$, $L$, and there exists a curve $\sigma_x:[0,1] \to X$ such that 
\begin{enumerate}
\item $\sigma_x(0) = g(\tau_x^2)$, $\sigma_x(1)=g(\tau_x^3)$, 
\item 
$\sigma_x([0,1]) \subset B(g(\tau_x^2),2d(g(\tau_x^2),g(\tau_x^3)))$, so
for each $t\in[0,1]$
\begin{align*} 
d(\sigma_x(t),\pi(x)) &\leq 2d(g(\tau_x^2),g(\tau_x^3)) + d(g(\tau_2^x),\pi(x))\\ 
&\leq 5(2^8\tilde L)^{-1}\xi\max\{\diam{\cL_x},\diam{\cR_x}\}
\end{align*}
\item $\len(\sigma_x) \leq  C' \max\{\diam{\cL_x},\diam{\cR_x}\}$
\item $\dist(\sigma_x([0,1]),g(\hat{A}\setminus (\hat{\cL}_x\cup\hat{\cR}_x))) \geq (C')^{-1}\min\{\diam{\cL_x},\diam{\cR_x}\}$.
\end{enumerate}

By Lemma \ref{Prop:Modi}, there exists $L^*>1$ depending only on $C$, $C_1$, $Q$, $p$, and $L$, and there exists an $L^*$-bi-Lipschitz map
\[ \gamma_x: [\tau_x^2,\tau_x^3] \to B\left(\pi(x),6(2^8\tilde L)^{-1}\xi\max\{\diam{\cL_x},\diam{\cR_x}\}\right) \] 
such that $\gamma_x(\tau_x^2)=\sigma_x(0)=g(\tau_x^2)$, $\gamma_x(\tau_x^3)=\sigma_x(1)=g(\tau_x^3)$, and for all $t\in [\tau_x^2,\tau_x^3]$,
\[ \dist(\gamma_x(t),\sigma_x([0,1])) \leq (2^{11}C'\tilde L)^{-1}\xi\max\{\diam{\cL_x},\diam{\cR_x}\}.\]
In particular, 
\begin{equation}\label{e-farmiddlepart}
\dist\left(\gamma_x([\tau_x^2,\tau_x^3]),g(\hat{A}\setminus (\cL_x\cup \cR_x))\right) \geq (2C')^{-1}\max\{\diam{\cL_x},\diam{\cR_x}\}.
\end{equation}

Set $\e = (2^{50}\tilde L\hat LL^*C')^{-2}\xi$.
Define
$$t_x^1 = \min \{ t \in [\tau_x^1,\tau_x^2] : \dist(g(t),\gamma_x([\tau_x^2,\tau_x^3])) = \e(\diam{\cL_x}+\diam{\cR_x}) \}$$
and
$$ t_x^2 = \max \{ t \in [\tau_x^2,\tau_x^3] : d(g(t_x^1),\gamma_x(t)) = \e(\diam{\cL_x}+\diam{\cR_x}) \}.$$

By \eqref{eq:neighbor}, Proposition~\ref{Prop:Endpoint}(3), and Proposition~\ref{Prop:middlethird}(1,2),
\begin{align*}
    d(g(\tau_x^1),g(\tau_x^2)) \geq d(\pi(x_L),\pi(x)) - d(\pi(x_L),g(\tau_x^1)) - d(\pi(x),g(\tau_x^2)) \geq \tfrac12\xi\diam{\cL_x},
\end{align*}
so we have that
\begin{align}\label{eq:est1}
d(g(t^1_x)&,g(\tau_x^1))\\ &\geq d(g(\tau_x^1),g(\tau_x^2)) - \max_{t\in [\tau_x^2,\tau_x^3]}d(\gamma_x(t),g(\tau_x^2)) - \dist(g(t_x^1),\gamma_x([\tau_x^2,\tau_x^3])) \notag\\
&\geq \tfrac12\xi\diam{\cL_x} - \max_{t\in[0,1]}d(\sigma_x(t),g(\tau_x^2))\notag\\ 
&\quad - (2^{11}C'\tilde L)^{-1} \xi\max\{\diam{\cL_x},\diam{\cR_x}\}- \e(\diam{\cL_x}+\diam{\cR_x})\notag \\
&\geq \left( \tfrac12\xi - 2^{-5} \xi - 2^{-9}\xi - 5\e \right)\diam{\cL_x}\notag \\
&\geq \tfrac14\xi\diam{\cL_x}\notag
\end{align}
and
\begin{align}\label{eq:est1'}
d(g(t^1_x),g(\tau_x^2)) \geq \dist(g(t_x^1),\gamma_x([\tau_x^2,\tau_x^3])) = \e(\diam{\cL_x}+\diam{\cR_x}).
\end{align}

Moreover,
\begin{align*}
d(\gamma_x(t_x^2),\gamma_x(\tau_x^3)) &\geq \dist(g(\tau_x^3),g([\tau_x^1,\tau_x^2])) - \dist(\gamma_x(t_x^2),g([\tau_x^1,\tau_x^2]))\\
&\geq \tfrac13\hat{L}^{-1}(\diam{\cL_x}+\diam{\cR_x}) - \e(\diam{\cL_x}+\diam{\cR_x})\\
&\geq \tfrac14\hat{L}^{-1}(\diam{\cL_x}+\diam{\cR_x}).
\end{align*}

Define now
\[ t_x^4 = \max\{t\in [\tau_x^3,\tau_x^4] : d(g(t),\gamma_x([\tau_x^2,\tau_x^3])) = \e(\diam{\cL_x}+\diam{\cR_x}) \}\]
and
$$ t_x^3 = \min \{ t \in [t_x^2,\tau_x^3] : d(\gamma_x(t), g(t_x^4)) = \e(\diam{\cL_x}+\diam{\cR_x}) \}.$$

\begin{figure}[ht]
    \centering
    \includegraphics[width=\textwidth]{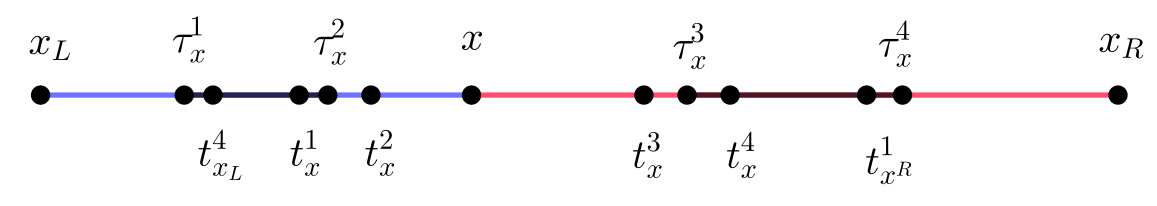}
    \caption{}
    \label{fig:2}
\end{figure}

As in \eqref{eq:est1}, we have that
\begin{equation}\label{eq:est2} 
d(g(t_x^4),g(\tau_x^4)) \geq \tfrac14\xi\diam{\cR_x}
\end{equation}
and
\begin{equation*}
d(g(t_x^4),g(\tau_x^3)) \geq \e(\diam{\cL_x}+\diam{\cR_x}).
\end{equation*}
Moreover, if $t\in [\tau_x^2,\tau_x^3]$ satisfies $d(\gamma_x(t),g(t_x^4)) = \e(\diam{\cL_x}+\diam{\cR_x})$, then
\begin{align*}
d(\gamma_x(t_x^2),\gamma_x(t)) &\geq \dist(g([\tau_x^1,\tau_x^2]),g([\tau_x^3,\tau_x^4])) - 2\e(\diam{\cL_x}+\diam{\cR_x})\\
&\geq (\tfrac13\hat{L}^{-1} - 2\e)(\diam{\cL_x}+\diam{\cR_x})\\
&\geq \tfrac14\hat{L}^{-1}(\diam{\cL_x}+\diam{\cR_x}).
\end{align*}
Therefore, $t_x^3$ is well defined and
\begin{equation}\label{eq:est3} 
t_x^3 - t_x^2 \geq (4\hat{L}L^*)^{-1}(\diam{\cL_x}+\diam{\cR_x}).
\end{equation}

\subsubsection{Local modifications around points in $E''$}
Fix $x\in E''$. We proceed to define $\gamma_x$ and points $t_x^1,\dots,t_x^4$ as in \textsection\ref{sec:E1}. The only difference is that we take into account the modifications done for points in $x_L,x_R \in E'$. In particular, we define
\begin{align*}
t_x^1 &= \min \{ t \in [t^4_{x_L},\tau_x^2] : \dist(g(t),\gamma_x([\tau_x^2,\tau_x^3])) = \e(\diam{\cL_x}+\diam{\cR_x}) \}\\
t_x^2 &= \max \{ t \in [\tau_x^2,\tau_x^3] : d(g(t_x^1),\gamma_x(t)) = \e(\diam{\cL_x}+\diam{\cR_x}) \}\\
t_x^4 &= \max\{t\in [\tau_x^3,t_{x_R}^1] : d(g(t),\gamma_x([\tau_x^2,\tau_x^3])) = \e(\diam{\cL_x}+\diam{\cR_x}) \}\\
t_x^3 &= \min \{ t \in [t_x^2,\tau_x^3] : d(\gamma_x(t), g(t_x^4)) = \e(\diam{\cL_x}+\diam{\cR_x}) \}.
\end{align*}
Equations \eqref{eq:est1}, \eqref{eq:est2}, \eqref{eq:est3} are still valid for $x \in E''$ as well. 

Furthermore, suppose that $x<y$ are consecutive points in $E$; that is, $x=y_L$ (or equivalently $y=x_R$). Then,
\begin{align}\label{eq:est4}
d(g&(t_x^4),g(t_{y}^1))\\ 
&\geq d(\pi(x),\pi(y)) - d(\pi(x),g(t_x^4)) - d(\pi(y),g(t_y^1)) \notag \\
&\geq \xi\diam{\cL_y} - \e(\diam{\cL_x}+\diam{\cR_x}) -6(2^8\tilde L)^{-1}\xi(\diam{\cL_x}+\diam{\cR_x})\notag\\ 
&\qquad -\e(\diam{\cL_y}+\diam{\cR_y}) -6(2^8\tilde L)^{-1}\xi(\diam{\cL_y}+\diam{\cR_y}) \notag\\
&\geq (\xi - 10\e - 60(2^8\tilde L)^{-1}\xi)\diam{\cL_y} \notag\\
&\geq \tfrac12\xi \diam{\cL_y}.\notag
\end{align}

\subsection{Definition of the extension \texorpdfstring{$F$}{F} and proof of Theorem \ref{thm:main2}}

Set
\[ \tilde A = \hat A \setminus \bigcup_{x \in E} [t_x^1,t_x^4]. \]
Define the map $F:I \to X$ so that
\begin{enumerate}
\item 
$F|\tilde A = g|\tilde A$,
\item for each $x\in E$, $F|[t_x^2,t_x^3] = \gamma_x|[t_x^2,t_x^3]$,
\item for each $x\in E$, $F|[t_x^1,t_x^2]$ is the geodesic from $g(t_x^1)$ to $\gamma_x(t_x^2)$ of constant speed,
\item for each $x\in E$, $F|[t_x^3,t_x^4]$ is the geodesic from $\gamma_x(t_x^3)$ to $g(t_x^4)$ of constant speed.
\end{enumerate}

Clearly, $F$ is an extension of $f$. In view of \eqref{eq:thm1}, the following proposition completes the proof of Theorem \ref{thm:main2}.

\begin{prop}
The map $F$ is an $L'$-bi-Lipschitz embedding for some $L'\geq 1$ depending only on $C$, $C_1$, $Q$, $p$, and $L$.
\end{prop}

\begin{proof}
Fix $s,t \in I$ with $s < t$. We may assume that one of $s,t$ is in $[t_x^1,t_x^4]$ for some $x \in E$, since, otherwise $F = g$, which is $\hat L$-bi-Lipschitz. Assume without loss of generality that $t \in [t_x^1,t_x^4]$ for some $x \in E$. The proof is a case study.

\emph{Case 1.} Assume that $s \in [t_x^1,t_x^4]$. There are a few subcases to consider.

\emph{Case 1.1.} Assume that $s,t \in [t_x^1,t_x^2]$ or $s,t \in [t_x^3,t_x^4]$. Without loss of generality, assume the former. In this case, $F(s)$ and $F(t)$ lie on a geodesic of unit speed joining $g(t_x^1)$ and $\gamma_x(t_x^2)$, and by \eqref{eq:est1'},
\[ \hat L^{-1}\e(\diam\cL_x + \diam\cR_x) \le 
|t_x^1 - \tau_x^2|
\leq
|t_x^1 - t_x^2| \le \diam\cL_x + \diam\cR_x, \]
so
\[ \frac{d(F(s),F(t))}{|s - t|} = \frac{d(g(t_x^1),\gamma_x(t_x^2))}{|t_x^1 - t_x^2|} \in [\e,\hat L]. \]

\emph{Case 1.2.} Assume that $s,t \in [t_x^2,t_x^3]$. Here $F|[t_x^2,t_x^3] = \gamma_x|[t_x^2,t_x^3]$, and $\gamma_x$ is $L^*$-bi-Lipschitz.

\emph{Case 1.3.} Assume that $s \in [t_x^1,t_x^2]$ and $t \in [t_x^2,t_x^3]$ or $s \in [t_x^2,t_x^3]$ and $t \in [t_x^3,t_x^4]$. Without loss of generality, we assume the former. Then $F(s)$ lies on a geodesic of unit speed joining $g(t_x^1)$ and $\gamma_x(t_x^2)$, and $F(t) = \gamma_x(t)$. Since $\gamma_x(t_x^2)$ is a closest point of $\gamma_x([t_x^2,t_x^3])$ to $g(t_x^1)$, Lemma~\ref{Lem:BL} implies that the gluing $F([t_x^1,t_x^3]) = g([t_x^1,t_x^2]) \cup \gamma_x([t_x^2,t_x^3])$ is bi-Lipschitz with a constant depending only on that of $\gamma_x$, which itself depends only on $C$, $C_1$, $Q$, $p$, and $L$.

\emph{Case 1.4.} Assume that $s \in [t_x^1,t_x^2]$ and $t \in [t_x^3,t_x^4]$. By \eqref{eq:est3},
\[ (4\hat LL^*)^{-1}(\diam\cL_x + \diam\cR_x) \le |t_x^2 - t_x^3| \le |s - t| \le \diam\cL_x + \diam\cR_x. \]
On one hand, using the fact that $F(s)$ and $F(t_x^2)$, and $F(t_x^3)$ and $F(t)$ lie on unit speed geodesics joining $g(t_x^1)$ to $\gamma_x(t_x^2)$ and $\gamma_x(t_x^3)$ to $g(t_x^4)$ respectively and
\begin{align*}
d(F(s),F(t)) & \leq d(F(s),F(t_x^2)) + d(\gamma_x(t_x^2),\gamma_x(t_x^3)) + d(F(t_x^3),F(t)) \\
& \leq d(g(t_x^1),\gamma_x(t_x^2)) + d(\gamma_x(t_x^2),\gamma_x(t_x^3)) + d(\gamma_x(t_x^3),g(t_x^4)) \\
& \le (2\e + 12(2^8\tilde L)^{-1}\xi)(\diam\cL_x + \diam\cR_x).
\end{align*}
On the other hand, arguing similarly gives
\begin{align*}
d(F(s),F(t)) & \ge d(\gamma_x(t_x^2),\gamma_x(t_x^3)) - d(F(s),F(t_x^2)) - d(F(t),F(t_x^3)) \\
& \ge (L^*)^{-1}|t_x^2 - t_x^3| - 2\e(\diam\cL_x + \diam\cR_x) \\
& \ge ((4\hat L(L^*)^2)^{-1} - 2\e)(\diam\cL_x + \diam\cR_x)\\
&\geq (8\hat{L}(L^*)^2)^{-1}(\diam\cL_x + \diam\cR_x).
\end{align*}

\emph{Case 2.} Assume that $s \in [t_y^1,t_y^4]$ for some $y \in E$ with $y < x$.
First, using \eqref{eq:est4},
\[ (10\hat L)^{-1}\xi(\diam\cR_y + \diam\cL_x) \le |t_y^4 - t_x^1| \le |s - t|. \]
As with Case 1.4,
\begin{align*}
d(F(s),F(t_y^4)) & \le d(F(t_y^1),F(t_y^2)) + d(\gamma_y(t_y^2),\gamma_y(t_y^3)) + d(F(t_y^3),F(t_y^4)) \\
& \le (2\e + 12(2^8\tilde L)^{-1}\xi)(\diam\cL_y + \diam\cR_y)
\end{align*}
and similarly
\[ d(F(t_x^1),F(t)) \le (2\e + 12(2^8\tilde L)^{-1}\xi)(\diam\cL_x + \diam\cR_x). \]
Thus
\begin{align*}
d(F(s),F(t)) & \le d(F(s),F(t_y^4)) + d(g(t_y^4),g(t_x^1)) + d(F(t_x^1),F(t)) \\
& \le 5(2\e + 1)(\diam\cR_y + \diam\cL_x) + \hat L|t_y^4 - t_x^1| \\
& \le 51\hat L\xi^{-1}(2\e + 1)|s - t|.
\end{align*}
For the lower bound, if $y = x_L$, then $|s-t| \leq 9|x-y|$ and \eqref{eq:est4} gives
\begin{align*}
    d(F(s),F(t)) &\ge d(g(t_y^4),g(t_x^1)) - d(F(s),F(t_y^4)) - d(F(t_x^1),F(t))\\
    &\ge \tfrac12 \xi \diam{\cL_x} - 10(2\e + 12(2^8\tilde L)^{-1}\xi)\diam\cL_x\\
    &\ge \tfrac{1}{50} \xi \diam{\cL_x}\\
    &\ge \tfrac{1}{450}|s-t|.
\end{align*} 
If instead $y < x_L$, then
\begin{align*}
d(F(s),F(t)) & \ge d(\pi(x),\pi(y)) - d(\pi(y),F(s)) - d(\pi(x),F(t)) \\
& \ge \tilde L^{-1}|x - y| - 5(12(2^8\tilde L)^{-1}\xi + 2\e)(\diam{\cR_y} + \diam{\cL_x}) \\
& \ge (\tilde L^{-1} - 10(12(2^8\tilde L)^{-1}\xi + 2\e))|x - y| \\
& \ge (18\tilde L)^{-1}|s - t|.
\end{align*}

\emph{Case 3.} Assume that $s \in \tilde{A}$. Then, $y\in [t_{y_L}^4,t_y^1]$ for some $y \in E$. There are two subcases to consider.

\emph{Case 3.1.} Assume that $y = x$. There are further subcases here.

\emph{Case 3.1.1} Assume that $t \in [t_x^1,t_x^2]$. As in Case 1.3, $g(t_x^1)$ is a closest point of $g([t_{x_L}^4,t_x^1])$ to $\gamma_x(t_x^2)$, so Lemma \ref{Lem:BL} tells us that $F([t_{x_L}^4,t_x^1])$ is bi-Lipschitz with a constant depending only on $C$, $C_1$, $Q$, $p$, and $L$.

\emph{Case 3.1.2.} Assume that $t \in [t_x^2,t_x^3]$. By \eqref{eq:est1'},
\[ \e\hat L^{-1}(\diam\cL_x + \diam\cR_x) \le |t_x^1 - \tau_x^2| \le |s - t| \le \diam\cL_x + \diam\cR_x, \]
so our desired bounds come from
\[ d(F(s),F(t)) \ge \dist(g([\tau_x^1,t_x^1]),\gamma_x([t_x^2,t_x^3])) = \e(\diam\cL_x + \diam\cR_x) \]
and
\begin{align*}
d(F(s),F(t)) & \le \diam g(\hat\cL_x) + \diam\gamma_x([\tau_2,\tau_3]) \\
& \le (\hat L + 12(2^8\tilde L)^{-1}\xi)(\diam\cL_x + \diam\cR_x).
\end{align*}

\emph{Case 3.1.3.} Assume that $t \in [t_x^3,t_x^4]$. By \eqref{eq:est3},
\[ (4\hat LL^*)^{-1}(\diam\cL_x + \diam\cR_x) \leq |t_x^2 - t_x^3| \le |t - s| \le \diam\cL_x + \diam\cR_x. \]
Now on one hand,
\begin{align*}
d(F(s),F(t)) & \le \diam g(\hat\cL_x) + \diam\gamma_x([\tau_x^2,\tau_x^3]) + \diam g(\hat\cR_x) \\
& \le (\hat L + 12(2^8\tilde L)^{-1}\xi)(\diam\cL_x + \diam\cR_x).
\end{align*}
On the other hand,
\begin{align*}
d(F(s),F(t)) & \ge \dist(g(\hat\cL_x),g(t_x^4)) - \diam F([t_x^3,t_x^4]) \\
& \ge ((3 \hat L)^{-1} - \e)(\diam\cL_x + \diam\cR_x).
\end{align*}

\emph{Case 3.2.} Assume that $y < x$. Then
\[ 3^{-1}(\diam\cR_y + \diam\cL_x) \le |\tau_y^2 - \tau_x^1| \le |s - t|. \]
As in Case 2, we have
\begin{align*}
d(F(s),F(t)) & \le d(g(s),g(t_y^1)) + d(g(t_y^1),g(t_x^1)) + d(F(t_x^1),F(t)) \\
& \le \hat L|s - t_y^1| + \hat L|t_y^1 - t_x^1| + (2\e + 12(2^8\tilde L)^{-1}\xi)(\diam\cL_x + \diam\cR_x) \\
& \le 3(\hat L + L^* + 2\e + 1)|s - t|.
\end{align*}
For the lower bound, set $M := (2\e + 6(2^8\tilde L)^{-1}\xi)$.
If $|s-t| \leq M\diam{\cL_x}$, then
the desired bound is a result of the following application of \eqref{e-farmiddlepart}:
\begin{align*}
    d(F(s),F(t)) 
    &\geq 
    \dist\left(F([t_x^1,t_x^4]),g(\hat{A}\setminus (\cL_x\cup \cR_x))\right) \\
    &\geq ((2C')^{-1} - 4\e) \max\{\diam{\cL_x},\diam{\cR_x}\}\\
    &\geq (16C')^{-1}\diam{\cL_x}.
\end{align*}
If $|s-t| > M\diam{\cL_x}$, then
\begin{align*}
d(F(s),F(t)) &\geq d(g(s), g(t_x^1)) - d(F(t_x^1),F(t))\\ 
&\geq \hat{L}^{-1}|s-t_x^1| - \diam{F([t_x^1,t_x^4])}\\ 
&\geq \tfrac1{16}\hat{L}^{-1}|s-t| - (2\e + 6(2^8\tilde L)^{-1}\xi)(\diam{\cL_x} + \diam{\cR_x})\\
&\geq \tfrac1{16}\hat{L}^{-1}|s-t| - 5(2\e + 6(2^8\tilde L)^{-1}\xi)\diam{\cL_x}\\
&\geq \tfrac1{32}\hat{L}^{-1}|s-t|. \qedhere
\end{align*}
\end{proof}

\subsection{The unbounded case}

Assuming that $X$ is unbounded, one can replace $I$ in Theorem \ref{thm:main2} by $\R$. The difference here is that we consider a Whitney decomposition of $\R\setminus A$. The unboundedness of $X$ guarantees the existence of function $\pi:E \to X$ as in Proposition \ref{Prop:Endpoint}. The rest of the proof is verbatim.

\bibliographystyle{alpha}
\bibliography{bibliography}
\end{document}